\documentclass[reqno,a4paper]{article}

\usepackage{a4wide}
\usepackage[centertags]{amsmath}
\usepackage{eucal,dsfont,accents,bbm,amsfonts,amssymb,amsthm,amsopn,amscd,units}
\usepackage[usenames]{color}
\definecolor{citegreen}{rgb}{0,0.6,0}
\definecolor{refred}{rgb}{0.8,0,0}
\usepackage[colorlinks, citecolor=citegreen, linkcolor=refred]{hyperref}

\title{The Chern-Gauss-Bonnet formula for singular non-compact four-dimensional manifolds}
\author{Reto Buzano and Huy The Nguyen}
\date{}

\providecommand{\abs}[1]{\lvert #1\rvert}

\newcommand{\Rm}{\mathrm{Rm}}
\newcommand{\Rc}{\mathrm{Rc}}
\newcommand{\R}{\mathrm{R}}
\newcommand{\RR}{\mathbb{R}}

\newcommand{\Sph}{\mathbb{S}}
\newcommand{\cH}{\mathcal{H}}

\newcommand{\tr}{\mathrm{tr}}
\newcommand{\eps}{\varepsilon}
\newcommand{\Lap}{\triangle}
\newcommand{\vol}{\mathrm{vol}}
\newcommand{\divop}{\mathrm{div}}
\newcommand{\dB}{\partial B}
\newcommand{\dt}{\frac{\partial}{\partial t}}
\newcommand{\ddt}{\frac{\partial^2}{\partial t^2}}

\def\Xint#1{\mathchoice
   {\XXint\displaystyle\textstyle{#1}}
   {\XXint\textstyle\scriptstyle{#1}}
   {\XXint\scriptstyle\scriptscriptstyle{#1}}
   {\XXint\scriptscriptstyle\scriptscriptstyle{#1}}
   \!\int}
\def\XXint#1#2#3{{\setbox0=\hbox{$#1{#2#3}{\int}$}
     \vcenter{\hbox{$#2#3$}}\kern-.5\wd0}}

\def\dashint{\Xint-}

\theoremstyle{plain}
\newtheorem{lemma}{Lemma}[section]

\newtheorem{thm}[lemma]{Theorem}
\newtheorem{theorem}[lemma]{Theorem}
\newtheorem{cor}[lemma]{Corollary}
\newtheorem{remark}[lemma]{Remark}
\newtheorem{defn}[lemma]{Definition}

\newtheorem{exam}[lemma]{Example}

\newtheorem{rem}{Remark\!\!}

\numberwithin{equation}{section}

\begin{document}
\maketitle
\begin{abstract}
We generalise the classical Chern-Gauss-Bonnet formula to a class of $4$-dimensional manifolds with finitely many conformally flat ends and singular points. This extends results of Chang-Qing-Yang in the smooth case. Under the assumptions of finite total $Q$ curvature and positive scalar curvature at the ends and at the singularities, we obtain a new Chern-Gauss-Bonnet formula with error terms that can be expressed as isoperimetric deficits. This is the first such formula in a dimension higher than two which allows the underlying manifold to have isolated branch points or conical singularities.
\end{abstract}


\section{Introduction}\label{sect.intro}

Relating the local geometry and global topology of manifolds constitutes one of the main aims of differential geometry. One of the most fundamental results is the Gauss-Bonnet formula
\begin{equation*}
\int_M K_g dV_g = 2\pi\chi(M),
\end{equation*}

which gives a link between the topology of a closed surface $(M,g)$ and its Gauss curvature. In particular, this formula yields topological obstructions to the existence of certain metrics, for example no torus $T^2 = \Sph^1 \times \Sph^1$ (with Euler characteristic zero) carries a metric $g$ of positive Gauss curvature $K_g>0$. It is well known that for complete non-compact surfaces or surfaces with singularities the Gauss-Bonnet formula requires additional terms known as isoperimetric deficit, which infinitesimally measure the deviation from flat Euclidean space. Such formulas have been extensively studied over the last eighty years, see for example~\cite{CV35,H57,H64,F65,JM,S85,LT,T91,CL95,N12} for some of the most important results.\\

A generalisation of the Gauss-Bonnet theorem to higher-dimensional compact Riemannian manifolds was discovered by Chern \cite{Ch43}. In particular, in the case of a compact $4$-dimensional manifold, the Chern-Gauss-Bonnet formula states that
\begin{equation}\label{eq.CGB}
32\pi^2\chi(M) = \int_M \big(\abs{\Rm_g}_g^2-4\abs{\Rc_g}_g^2+\R_g^2\big)dV_g = \int_M \big(\abs{W_g}_g^2+8Q_g\big) dV_g,
\end{equation}

where $\Rm_g$, $\Rc_g$, $R_g$ and $W_g$ denote the Riemannian, Ricci, scalar and Weyl curvature of $(M^4,g)$, respectively, and 
\begin{equation}\label{eq.Q}
Q_g:=-\tfrac{1}{12}\big(\Lap_g \R_g -\R_g^2 +3\abs{\Rc_g}_g^2\big)
\end{equation}

is the Paneitz $Q$ curvature introduced by Branson \cite{B85}. Under a conformal change $g=e^{2w}g_0$, this scalar quantity transforms by
\begin{equation}\label{eq.Paneitz1}
P_{g_0}w+2Q_{g_0}=2Q_g e^{4w},
\end{equation}

where $P_g$ denotes the Paneitz operator, introduced in \cite{P83} and defined by 
\begin{equation}\label{eq.P}
P_g(\varphi):= \Lap_g^2\varphi+\divop_g\big(\tfrac{2}{3}\R_g g-2\Rc_g\big)d\varphi, \qquad \varphi\in C^\infty(M).
\end{equation}

Under the same conformal change as above, the Paneitz operator transforms by
\begin{equation}
P_g=e^{-4w}P_{g_0}.
\end{equation}

These transformation laws are the higher-dimensional equivalents of the classical formulas for surfaces, stating that for $g=e^{2w}g_0$ we have
\begin{equation*}
-\Lap_{g_0} w + K_{g_0} = K_g e^{2w}, \qquad \Lap_g = e^{-2w}\Lap_{g_0}.
\end{equation*}

For more details on the properties of the Paneitz operator and the $Q$ curvature, we refer the reader to \cite{P83,B85,BCY,CQ}.\\

If $M$ is non-compact or singular, very little is known about how the Chern-Gauss-Bonnet formula should look like. In the smooth case, besides the classical results of Cheeger-Gromov on manifolds with bounded geometry \cite{CG} and the results of Greene-Wu on complete four-manifolds with positive sectional curvature \cite{GW}, the most general results are due to Chang-Qing-Yang \cite{CQY1,CQY2}. (See also \cite{Fa05,NX} for higher-dimensional versions of similar results.) In the singular case, this problem has only been studied in the special case of so-called edge-cone singularities \cite{LS98,AL} or $V$-manifolds \cite{Sat57}, but no results seem to be known that allow the underlying manifold to have isolated singular points.\\

The goal of this article is to develop a formula for a large class of manifolds $M$ which are diffeomorphic to a compact manifold with finitely many points removed, allowing in particular both complete ends and finite area interior branch points. For the reader's convenience and simplicity of exposition, we first state the result in the simple situation of a conformal metric $e^{2w}\abs{dx}^2$ on $\RR^4\setminus\{0\}$ having one complete end at infinity and one finite area singular point at the origin (see Theorem \ref{thm.R4} for the precise assumptions). This is a model case for the more general situation of manifolds with many conformally flat ends and singular points, see Definition \ref{def.manifold} and Theorem \ref{thm.manifold}.

\begin{thm}[Chern-Gauss-Bonnet formula for singular metrics conformal to $\RR^4$]\label{thm.R4}
Let $g=e^{2w}\abs{dx}^2$ be a metric on $\RR^4\setminus\{0\}$ which is complete at infinity and has finite area over the origin. If $g$ has finite total $Q$ curvature
\begin{equation}\label{eq.finiteQ}
\int_{\RR^4}\abs{Q_g}\, dV_g=\int_{\RR^4}\abs{Q_g}\, e^{4w}\,dx<\infty,
\end{equation}
and non-negative scalar curvature at infinity and at the origin, i.e.
\begin{equation}\label{eq.nonnegR}
\inf_{\RR^4\setminus B_{r_2}(0)} \R_g(x)\geq 0,\quad \inf_{B_{r_1}(0)} \R_g(x)\geq 0,
\end{equation}
for some $0<r_1\leq r_2<\infty$, then we have
\begin{equation}\label{eq.main}
\chi(\RR^4)-\frac{1}{4\pi^2}\int_{\RR^4}Q_g e^{4w}\,dx = \nu-\mu,
\end{equation}
where
\begin{equation}\label{eq.nu1}
\nu :=\lim_{r\to\infty}\, \frac{\vol_g(\dB_r(0))^{4/3}}{4(2\pi^2)^{1/3}\vol_g(B_r(0))},\qquad \mu :=\lim_{r\to 0}\, \frac{\vol_g(\dB_r(0))^{4/3}}{4(2\pi^2)^{1/3}\vol_g(B_r(0))}-1.
\end{equation}
\end{thm}

The simplest and at the same time fundamental example that illustrates this result and that is not covered by previously known Chern-Gauss-Bonnet type formulas is the following conical metric.

\begin{exam}
Let us consider the manifold $(\RR^4 \setminus \{0\},g^\alpha)$, where the metric $g^\alpha$ is defined by $g^\alpha_{ij} = \delta _{ij} + \alpha\frac{x_ix_j}{r^2}$ for some $-1<\alpha< \infty$. Here, $r^2 = \abs{x}^2$, where $\abs{x}$ denotes the Euclidean norm. If $\alpha=0$, this is simply Euclidean space. More generally, the metric $g^\alpha$ is a conformal deformation of the Euclidean metric, as can be seen as follows. First note that in spherical co-ordinates, the metric may be written as 
\begin{align*}
g^\alpha=(1+\alpha)dr^2 + r^2 d\sigma_{\Sph^3}. 
\end{align*}    
Hence, we can reparametrise with $s = r^{\sqrt{1+\alpha}}$ so that 
\begin{equation*}
dr^2=\tfrac{1}{1+\alpha}s^{2(\frac{1}{\sqrt{1+\alpha}}-1)}ds^2, \qquad r^2=s^{\frac{2}{\sqrt{1+\alpha}}}
\end{equation*}
and the metric becomes
\begin{equation*}
g^\alpha = s^{2(\frac{1}{\sqrt{1+\alpha}}-1)}\Big(ds^2 + s^2 d\sigma_{\Sph^3}\Big)= s^{2(\frac{1}{\sqrt{1+\alpha}}-1)} g_{\RR^4}.
\end{equation*}
Hence $g^\alpha$ is a complete metric conformal to $g_{\RR^4}$ with conformal factor
\begin{align*}
e^{2w(r)} = r^{2(\frac{1}{\sqrt{1+\alpha}} -1)}
\end{align*}
so that $w(r) = (\frac{1}{\sqrt{1+\alpha}} -1) \log r$. Computing, we see that
\begin{align*}
\nu =\lim_{r\to\infty}\, \frac{\vol_g(\dB_r(0))^{4/3}}{4(2\pi^2)^{1/3}\vol_g(B_r(0))} = \frac{1}{\sqrt{1 + \alpha}}, \quad
\mu = \frac{1}{\sqrt{1+\alpha}}-1.
\end{align*}
We note that for this cone in dimension four,
\begin{align*}
\R_{g^\alpha} =  \frac{6\alpha e^{-2w(r)}}{(1+\alpha)r^2}, \quad \abs{\Rc_{g^\alpha}}_{g^\alpha}^2 =  \frac{12\alpha^2 e^{-4w(r)}}{(1+\alpha)^2 r^4}, \quad \Lap_{g^\alpha} R_{g^\alpha} = 0. 
\end{align*}
so that $Q_{g^\alpha} = 0.$ Hence in this case
\begin{align*}
1 = \chi(\RR^4) - \frac{1}{4\pi^2} \int_{\RR^4} Q_{g^\alpha}\; dV_{g^\alpha}= \nu - \mu.
\end{align*}
\end{exam}

Theorem \ref{thm.R4} above shows that the deficit in the Chern-Gauss-Bonnet formula \eqref{eq.main} is given by the limits of some isoperimetric ratios. In contrast to the situation on surfaces, there are several different isoperimetric ratios in higher dimensions. Recall that for $\Omega$ a convex domain in $\RR^n$ and $1\leq m<n$, the mixed volumes $V_m(\Omega)$, as defined by Trudinger \cite{T97}, are given by
\begin{equation}
V_m(\Omega)=\frac{1}{n\binom{n-1}{m}}\int_{\partial\Omega} H_{n-1-m}\,[\partial\Omega]\,d\cH^{n-1},
\end{equation}

where $H_{k}\,[\partial\Omega]$ denotes the $k$-th symmetric polynomial in the principal curvatures of $\partial\Omega$ and $\cH^{n-1}$ is the $(n-1)$-dimensional Hausdorff measure. In particular, if $\Omega=B_r(0)$ is a ball in $\RR^4$ centred at the origin, we obtain the following mixed volumes and isoperimetric ratios.

\begin{defn}\label{defn.ratios}
We define the volumes $V_k(r)$ by
\begin{equation}\label{eq.Vs1}
\begin{aligned}
& V_4(r) =\int_{B_r(0)} e^{4w}\, dx,\qquad && V_3(r) =\frac{1}{4}\int_{\dB_r(0)} e^{3w}\, d\sigma(x),\\
& V_2(r) =\frac{1}{12}\int_{\dB_r(0)}H_1\, e^{3w}\, d\sigma(x),\qquad && V_1(r) =\frac{1}{12}\int_{\dB_r(0)}H_2\, e^{3w}\, d\sigma(x).
\end{aligned}
\end{equation}
The isoperimetric ratios $C_{k,\ell}(r)$ are then defined by
\begin{equation}\label{eq.Cs1}
\begin{aligned}
C_{3,4}(r)&=\frac{V_3^{4/3}(r)}{(\pi^2/2)^{1/3}\,V_4(r)},\\
C_{2,3}(r)&=\frac{V_2(r)}{(\pi^2/2)^{1/3}\,V_3^{2/3}(r)},\\
C_{1,2}(r)&=\frac{V_1^{2/3}(r)}{(\pi^2/2)^{1/3}\,V_2^{1/3}(r)},
\end{aligned}
\end{equation}
and
\begin{equation}\label{eq.Cs2}
\begin{aligned}
C_{2,4}(r)&=C_{3,4}^{1/3}(r)\cdot C_{2,3}^{2/3}(r),\\
C_{1,3}(r)&=C_{2,3}^{1/4}(r)\cdot C_{1,2}^{3/4}(r),\\
C_{1,4}(r)&=C_{3,4}^{1/9}(r)\cdot C_{2,3}^{2/9}(r)\cdot C_{1,2}^{2/3}(r).\\
\end{aligned}
\end{equation}
\end{defn}

Our second result says that asymptotically, all these isoperimetric volume ratios agree in our setting.

\begin{thm}[Isoperimetric volume ratios agree asymptotically]\label{thm.ratios}
Let $g$ be a metric as in Theorem \ref{thm.R4}, and let $\mu$ be defined as in \eqref{eq.nu1}. If 
\begin{equation*}
1 + \mu - \frac{1}{4\pi^2} \int_{\RR^4} Q_g e^{4w} dx > 0
\end{equation*}
then the Chern-Gauss-Bonnet formula \eqref{eq.main} holds for $\nu$ given by
\begin{equation}\label{eq.nu2}
\nu :=\lim_{r\to\infty}\, C_{k,\ell}(r),
\end{equation}
for \emph{any} choice of $1\leq k<\ell\leq 4$.
Alternatively, if $\nu$ is defined as in \eqref{eq.nu1} and satisfies 
\begin{equation*} 
1 - \nu + \frac {1}{4\pi^2} \int_{\RR^4} Q_g e^{4w} dx > 0
\end{equation*}
then the Chern-Gauss-Bonnet formula \eqref{eq.main} holds for $\mu$ given by
\begin{equation}\label{eq.nu3}
\mu := \lim_{r \to 0} C_{k,\ell}(r)-1.
\end{equation}
for \emph{any} choice of $1\leq k<\ell\leq 4$.
\end{thm}

\begin{rem}
Of course, picking $k=3$ and $\ell=4$ in \eqref{eq.nu2} or \eqref{eq.nu3}, we obtain \eqref{eq.nu1}.
\end{rem}

In order to state the main result, generalising Theorem \ref{thm.R4} to the case of $4$-manifolds with finitely many ends and singular points, let us first explain which type of manifolds we are exactly considering here.
\begin{defn}\label{def.manifold}
We say $(M,g)$ is a $4$-manifold with finitely many conformally flat complete ends and finite-area singular points, if
\begin{equation*}
M=N\cup\Big(\bigcup_{i=1}^k E_i\Big)\cup\Big(\bigcup_{j=1}^{\ell} S_j\Big)
\end{equation*}
where $(N,g)$ is a compact Riemannian manifold with boundary
\begin{equation*}
\partial N=\Big(\bigcup_{i=1}^k \partial E_i\Big)\cup\Big(\bigcup_{j=1}^{\ell} \partial S_j\Big)
\end{equation*}
and the $E_i$, $S_j$ satisfy the following properties. Each $E_i$ is a conformally flat simple end, meaning that
\begin{equation}\label{eq.ends}
(E_i,g)=(\RR^4\setminus B,e^{2e_i}\abs{dx}^2)
\end{equation}
for some function $e_i(x)$, where $B$ is the unit ball in $\RR^4$ and the metric is complete at infinity. Each $S_j$ on the other hand is a conformally flat region with finite area and with a point-singularity at some $p_j$, meaning that
\begin{equation}\label{eq.sings}
(S_j\setminus\{p_j\},g)=(B\setminus\{0\},e^{2s_j}\abs{dx}^2)
\end{equation}
for some function $s_j(x)$, where $B$ again denotes the unit ball in $\RR^4$.
\end{defn}

The main result of this article states the following.
\begin{thm}[Chern-Gauss-Bonnet formula for singular non-compact $4$-manifolds]\label{thm.manifold}
Let $(M,g)$ be as in Definition \ref{def.manifold} and assume that $g$ has finite total $Q$ curvature
\begin{equation}\label{eq.118}
\int_{M}\abs{Q_g}\, dV_g<\infty,
\end{equation}
and non-negative scalar curvature at every singular point and at infinity at each end. Then we have
\begin{equation}\label{eq.mainCGBformula}
\chi(M)-\frac{1}{32\pi^2}\int_M \big(\abs{W_g}_g^2+8Q_g\big) dV_g = \sum_{i=1}^k\nu_i-\sum_{j=1}^{\ell}\mu_j,
\end{equation}
where
\begin{equation}
\nu_i :=\lim_{r\to\infty}\, \frac{\big(\int_{\dB_r(0)}e^{3e_i(x)}d\sigma(x)\big)^{4/3}}{4(2\pi^2)^{1/3}\int_{B_r(0)\setminus B}e^{4e_i(x)}dx},\qquad i=1,\ldots,k,
\end{equation}
and
\begin{equation}
\mu_j :=\lim_{r\to0}\, \frac{\big(\int_{\dB_r(0)}e^{3s_j(x)}d\sigma(x)\big)^{4/3}}{4(2\pi^2)^{1/3}\int_{B_r(0)}e^{4s_j(x)}dx}-1,\qquad j=1,\ldots,\ell.
\end{equation}
\end{thm}

\begin{rem}
In the case of \emph{smooth} metrics satisfying \eqref{eq.118} and the condition of positive scalar curvature at each end, Theorem \ref{thm.manifold} reduces to the results of Chang-Qing-Yang \cite{CQY1,CQY2}. Moreover, under these assumptions, Wang obtained interesting isoperimetric inequalities, see \cite{W1,W2}. It would be interesting to try to adopt her results to our situation of manifolds with singularities.
\end{rem}

The paper is organised as follows. In the Sections \ref{sect.symm}--\ref{sect.general}, we prove the Theorems \ref{thm.R4} and \ref{thm.ratios} in three steps as follows. First, note that for a conformal metric $g=e^{2w}\abs{dx}^2$ on $\RR^4\setminus\{0\}$, the definition of the Paneitz operator \eqref{eq.P} together with the Paneitz equation \eqref{eq.Paneitz1} imply
\begin{equation}\label{eq.Paneitz2}
\Lap^2 w=2Q_g e^{4w},
\end{equation}

where $\Lap$ denotes the Euclidean Laplacian. In the special case where $w=w(r)$ is a radial function on $\RR^4\setminus\{0\}$, this equation reduces to an ODE. In Section \ref{sect.symm}, by solving explicitly for the non-linearity in this ODE, we prove the two theorems for the special case of rotationally symmetric metrics. Then, we introduce a new notion of generalised normal metrics, namely metrics of the form $g=e^{2w}\abs{dx}^2$ where
\begin{equation}\label{eq.normal}
w(x)=\frac{1}{4\pi^2}\int_{\RR^4}\log\Big(\frac{\abs{y}}{\abs{x-y}}\Big)\, Q(y)\,e^{4w(y)}\,dy+\alpha\log\abs{x}+C.
\end{equation}

In Section \ref{sect.normal}, we prove Theorem \ref{thm.R4} and Theorem \ref{thm.ratios} for this type of metrics by comparing $w$ satisfying \eqref{eq.normal} with the averaged conformal factor 
\begin{equation}\label{eq.averaged}
\bar{w}(r):=\frac{1}{\abs{\dB_r(0)}}\int_{\dB_r(0)}w(x)\,d\sigma(x)
\end{equation}
and reducing this case to the already established rotationally symmetric one. Next, in Section \ref{sect.general}, we show that every metric $g=e^{2w}\abs{dx}^2$ on $\RR^4\setminus\{0\}$ which is complete at infinity and has finite area over the origin and satisfies \eqref{eq.finiteQ} and \eqref{eq.nonnegR} is a generalised normal metric. This finishes the proofs of Theorem \ref{thm.R4} and Theorem \ref{thm.ratios} in the general case. Finally, in Section \ref{sect.local}, we localise Theorem \ref{thm.R4} to metrics as in \eqref{eq.ends} and \eqref{eq.sings} and prove Theorem \ref{thm.manifold} by gluing together all the pieces.


\section{Rotationally symmetric metrics}\label{sect.symm}

In this section, we assume that $g=e^{2w}\abs{dx}^2$ is a conformal metric on $\RR^4\setminus\{0\}$ and $w=w(r)$ is a \emph{radial} function. Using spherical coordinates and writing $\abs{x}=r=e^t$, one obtains $\Lap w=e^{-2t}(\ddt+2\dt)w$, and hence, using also the Paneitz equation \eqref{eq.Paneitz2},
\begin{equation}\label{eq.Paneitz3}
\Lap^2 w=e^{-4t}\Big(\ddt-2\dt\Big)\Big(\ddt+2\dt\Big)w=2Q_ge^{4w}, \qquad -\infty<t<\infty.
\end{equation}

Note that $t=\log r$ satisfies $\Lap^2 t=0$. Therefore, it makes sense to denote $w+t$ by $v$, which solves
\begin{equation}\label{eq.ODE}
v''''-4v''=2Q_ge^{4v}, \qquad -\infty<t<\infty.
\end{equation}

The goal of the first part of this section is to compute the limits of $v'(t)$ for $t\to\pm\infty$. In order to do so, following \cite{CQY1}, we denote $F:=2Q_g e^{4v}$ and construct an explicit solution $f(t)$ of the equation 
\begin{equation}\label{eq.fODE}
f''''-4f''=F.
\end{equation}

It will turn out that we do not need $f(t)$ itself, but only its derivatives. Before explaining the construction of $f(t)$, let us prove a technical lemma.

\begin{lemma}\label{lemma.K}
For $F=2Q_g e^{4v}$ as above, we have
\begin{equation*}
K_1:=\lim_{t\to-\infty}e^{2t}\int_t^{\infty} F(x)e^{-2x}dx=0,\qquad
K_2:=\lim_{t\to\infty}e^{-2t}\int_{-\infty}^t F(x)e^{2x}dx=0,
\end{equation*}
\end{lemma}

\begin{proof}
First note that 
\begin{equation}
\int_{-\infty}^{\infty} \abs{F(x)}dx =\int_{-\infty}^{\infty}2\abs{Q_g}e^{4v}dx=\frac{2}{\abs{\Sph^3}} \int_{\RR^4}\abs{Q_g}e^{4w}dx=\frac{1}{\pi^2} \int_{\RR^4}\abs{Q_g}dV_g<\infty
\end{equation}
by assumption. We then have for $t<T<\infty$
\begin{align*}
\left| e^{2t}\int_t^{\infty} F(x)e^{-2x}dx \right|&\leq e^{2t}\Big(\int_t^T\abs{F(x)}e^{-2x}dx+\int_T^\infty\abs{F(x)}e^{-2x}dx\Big)\\
&\leq \int_{-\infty}^T \abs{F(x)}dx +e^{2(t-T)} \int_{-\infty}^{\infty}\abs{F(x)}dx.
\end{align*}
Setting $T=t/2$ (for negative $t$) and letting $t$ tend to $-\infty$, we obtain $K_1=0$. Similarly, we have for $-\infty<T<t$
\begin{align*}
\left| e^{-2t}\int_{-\infty}^t F(x)e^{2x}dx \right|&\leq e^{-2t}\Big(\int_{-\infty}^T\abs{F(x)}e^{2x}dx+\int_T^t\abs{F(x)}e^{2x}dx\Big)\\
&\leq e^{2(T-t)}\int_{-\infty}^{\infty} \abs{F(x)}dx + \int_T^{\infty}\abs{F(x)}dx.
\end{align*}
Hence, setting $T=t/2$ (for positive $t$) and letting $t$ tend to $\infty$, we obtain $K_2=0$.
\end{proof}

A consequence of this lemma (which could also be seen more directly) is that 
\begin{equation*}
K_3:=\lim_{t\to-\infty}e^{4t}\int_t^{\infty} F(x)e^{-2x}dx=0,\qquad
\end{equation*}

In order to find $f(t)$, we make the ansatz $f''(t)=C(t)e^{-2t}$. Plugging this into \eqref{eq.fODE} yields
\begin{equation*}
(C'(t)e^{-4t})'=F(t)e^{-2t},
\end{equation*}

which can be solved for $C(t)$ as follows:
\begin{align*}
C(t)&=-\int_{-\infty}^t e^{4x}\bigg(\int_x^{\infty} F(y)e^{-2y}dy\bigg)dx\\
&=-\frac{1}{4}e^{4t}\int_t^{\infty} F(x)e^{-2x}dx -\frac{1}{4}\int_{-\infty}^t F(x)e^{2x}dx,
\end{align*}

where the second line is obtained using integration by parts and $K_3=0$. We thus find
\begin{equation*}
f''(t)=-\frac{1}{4}e^{2t}\int_t^{\infty}F(x)e^{-2x}dx-\frac{1}{4}e^{-2t}\int_{-\infty}^t F(x)e^{2x}dx.
\end{equation*}

By Lemma \ref{lemma.K}, we have
\begin{equation}\label{eq.limitsecondderiv}
\lim_{t\to -\infty}f''(t)=\lim_{t\to\infty}f''(t)=0.
\end{equation}

Another integration by parts, using $K_1=K_2=0$, yields
\begin{align*}
f'(t)&=-\int_{-\infty}^t \frac{1}{4}e^{2x}\bigg(\int_x^{\infty} F(y)e^{-2y}dy\bigg)dx+\int_t^{\infty}\frac{1}{4}e^{-2x}\bigg(\int_{-\infty}^x F(y)e^{2y}dy\bigg)dx\\
&=-\frac{1}{8}e^{2t}\int_t^{\infty} F(x)e^{-2x}dx -\frac{1}{8}\int_{-\infty}^t F(x)dx\\
&\quad +\frac{1}{8}e^{-2t}\int_{-\infty}^t F(x)e^{2x}dx+\frac{1}{8}\int_t^{\infty} F(x)dx.
\end{align*}

Using Lemma \ref{lemma.K} once more, we obtain
\begin{equation}\label{eq.limitfirstderiv}
\lim_{t\to -\infty}f'(t)=\frac{1}{8}\int_{-\infty}^{\infty}F(x)dx,\qquad \lim_{t\to\infty}f'(t)=-\frac{1}{8} \int_{-\infty}^{\infty}F(x)dx.
\end{equation}

We could now obtain an explicit $f(t)$ by a further integration and requiring for instance $f(0)=0$. However, we only need the derivatives of $f(t)$, hence we skip this step. Our $v(t)$ is now of the form
\begin{equation}\label{eq.v}
v(t)=c_0+c_1 t+c_2 e^{-2t}+c_3 e^{2t} + f(t)
\end{equation}

for some constants $c_0$, $c_1$, $c_2$, and $c_3$. We first prove that the assumption of asymptotically non-negative scalar curvature implies that $c_2$ and $c_3$ vanish.

\begin{lemma}\label{lemma.c23}
Assume that $g=e^{2w}\abs{dx}^2$ has non-negative scalar curvature at infinity and at the origin, i.e.
\begin{equation*}
\inf_{\RR^4\setminus B_{r_2}(0)} \R_g(x)\geq 0,\quad \inf_{B_{r_1}(0)} \R_g(x)\geq 0,
\end{equation*}
for some $0<r_1\leq r_2<\infty$. Then $c_2=c_3=0$ in \eqref{eq.v}.
\end{lemma}

\begin{proof}
The transformation law for scalar curvature under a conformal change $g=e^{2w}g_0$ on an $n$-dimensional manifold is
\begin{equation*}
\R_g = e^{-2w}\Big(\R_{g_0}-\tfrac{4(n-1)}{n-2}e^{-\frac{n-2}{2}w}\Lap_{g_0}e^{\frac{n-2}{2}w}\Big).
\end{equation*}
In our case, where $g_0$ is the Euclidean metric and $n=4$, this
becomes
\begin{equation*}
\tfrac{1}{6}\R_g e^{2w}=-e^{-w}\Lap e^w=-e^{-2t}e^{-w}\Big(2\dt+\ddt\Big)e^w,
\end{equation*}
or equivalently
\begin{equation}\label{eq.Rv}
\tfrac{1}{6}\R_g e^{2v(t)}=-v''(t)-(v'(t))^2+1.
\end{equation}
By our assumption of asymptotically non-negative scalar curvature at infinity and at the origin, the left hand side of \eqref{eq.Rv} becomes non-negative if $t\to\pm\infty$. On the other hand, we know by \eqref{eq.limitsecondderiv} and \eqref{eq.limitfirstderiv} that $f''(t)$ and $f'(t)$ stay bounded as $t\to\pm\infty$, thus the dominating term on the right hand side of \eqref{eq.Rv} is $-4\abs{c_3}^2e^{4t}$ as $t\to\infty$, respectively $-4\abs{c_2}^2e^{-4t}$ as $t\to-\infty$. Hence the right hand side of \eqref{eq.Rv} can only be non-negative if $c_2=c_3=0$.
\end{proof}

\begin{rem}
The proof shows that it is sufficient to assume that the term $\R_g e^{2v}=\abs{x}\R_g e^{2w}$ is bounded from below, rather than assuming the stronger condition that $\R_g\geq 0$ near infinity and the origin. On the other hand, our assumption \eqref{eq.nonnegR} cannot be replaced by the weaker assumption $\R_g\to 0$ as $t=\log\abs{x}\to\pm\infty$. This is for example illustrated by the metric $g=e^{2r^2}\abs{dx}^2$, which is complete at infinity, satisfies $Q_g\equiv 0$ and 
\begin{equation*}
\R_g(r)=-48e^{-2r^2}-24r^2e^{-2r^2}\to 0 \quad (\text{as }r\to\infty),
\end{equation*}
but has $c_3=1$.
\end{rem}

A direct consequence of Lemma \ref{lemma.c23} is that under the assumptions as above, we have 
\begin{equation*}
\lim_{t\to\infty}v'(t)-\lim_{t\to-\infty}v'(t)=-\frac{1}{4} \int_{-\infty}^{\infty}F(x)dx=-\frac{1}{4\pi^2} \int_{\RR^4}Q_g e^{4w}dx.
\end{equation*}

This follows from $c_2=c_3=0$, which implies $v(t)=c_0+c_1 t+f(t)$ and thus $v'(t)=c_1+f'(t)$, and the computations of the limits of $f'(t)$ in \eqref{eq.limitfirstderiv}. We have thus proved the following.

\begin{cor}\label{cor.CGBv}
Under the assumptions as above, we have
\begin{equation*}
\chi(\RR^4)-\frac{1}{4\pi^2}\int_{\RR^4}Q_g e^{4w}\,dx = \nu-\mu,
\end{equation*}
where
\begin{equation*}
\nu :=\lim_{t\to\infty}v'(t),\qquad \mu :=\lim_{t\to-\infty}v'(t)-1.
\end{equation*}
\end{cor}

To finish the proofs of both Theorem \ref{thm.R4} and Theorem \ref{thm.ratios} in the rotationally symmetric case, we only need to prove \eqref{eq.nu2} and \eqref{eq.nu3}, i.e. we need to relate $v'(t)$ to the isoperimetric volume ratios.

\begin{lemma}\label{lemma.limits}
We have
\begin{align*}
\lim_{t\to\infty}v'(t)=\lim_{t\to\infty}C_{k,\ell}(e^t),\\
\lim_{t\to-\infty}v'(t)=\lim_{t\to-\infty}C_{k,\ell}(e^t),\\
\end{align*}
for any choice of $1\leq k<\ell\leq 4$.
\end{lemma}

\begin{proof}
Letting $H_1$, $H_2$ denote the first and second symmetric polynomial in the principal curvatures of $\dB_r(0)$ as in \eqref{eq.Vs1}, a short computation yields
\begin{equation*}
H_1=3e^{-w}\Big(\frac{1}{r}+\frac{\partial w}{\partial r}\Big), \qquad H_2=3e^{-2w}\Big(\frac{1}{r}+\frac{\partial w}{\partial r}\Big)^2, 
\end{equation*}
and hence
\begin{equation}\label{eq.Vs2}
V_2(r) =\frac{1}{4}\int_{\dB_r(0)}\Big(\frac{1}{r}+\frac{\partial w}{\partial r}\Big)e^{2w}\, d\sigma(x),\quad 
V_1(r) =\frac{1}{4}\int_{\dB_r(0)}\Big(\frac{1}{r}+\frac{\partial w}{\partial r}\Big)^2e^{w}\, d\sigma(x).
\end{equation}
Substituting $r=e^t$ (which implies $dx=e^{4t} dt d\sigma_{\Sph^3}$ as well as $d\sigma(x)= e^{3t}d\sigma_{\Sph^3}$ on $\dB_r(0)$) and $v=w+t$, we obtain from \eqref{eq.Vs1} and \eqref{eq.Vs2}
\begin{equation}\label{eq.Vs3}
\begin{aligned}
V_4(r)&=V_4(e^t)=\int_{-\infty}^{t} \int_{\Sph^3} e^{4v(s)} d\sigma_{\Sph^3}ds = \abs{\Sph^3}\int_{-\infty}^{t}e^{4v(s)}ds,\\
V_3(r)&=V_3(e^t)=\frac{1}{4}\int_{\Sph^3}e^{3v(t)}d\sigma_{\Sph^3}=\frac{1}{4}\abs{\Sph^3}e^{3v(t)},\\
V_2(r)&=V_2(e^t)=\frac{1}{4}\int_{\Sph^3}v'(t)e^{2v(t)}d\sigma_{\Sph^3}=\frac{1}{4}\abs{\Sph^3}v'(t)e^{2v(t)},\\
V_1(r)&=V_1(e^t)=\frac{1}{4}\int_{\Sph^3}(v'(t))^2e^{v(t)}d\sigma_{\Sph^3}=\frac{1}{4}\abs{\Sph^3}(v'(t))^2e^{v(t)}.
\end{aligned}
\end{equation}
It now follows directly from Definition \ref{defn.ratios} that
\begin{equation}\label{eq.Candv}
C_{2,3}(e^t)=C_{1,2}(e^t)=C_{1,3}(e^t)=v'(t).
\end{equation}
This proves the lemma for all cases where $\ell\neq 4$. We thus turn to study $C_{3,4}(e^t)$. As $g=e^{2w}\abs{dx}^2$ is complete at infinity by assumption, we conclude that $\lim_{t\to\infty}v'(t)\geq 0$. If this limit is strictly positive, then both $V_4(e^t)$ and $V_3(e^t)$ tend to infinity as $t\to\infty$ and we obtain from L'H\^{o}pital's rule
\begin{equation*}
\lim_{t\to\infty}C_{3,4}(e^t) =\lim_{t\to\infty}\frac{\frac{4}{3}(\frac{1}{4}\abs{\Sph^3}e^{3v(t)})^{1/3}\cdot\frac{3}{4}\abs{\Sph^3}v'(t)e^{3v(t)}}{(\pi^2/2)^{1/3}\abs{\Sph^3}e^{4v(t)}}=\lim_{t\to\infty}v'(t).
\end{equation*}
On the other hand, if $\lim_{t\to\infty}v'(t)= 0$, and  $V_4(e^t)$ stays bounded as $t\to\infty$, then $\lim_{t\to\infty}e^{4v(t)}=0$, which implies $\lim_{t\to\infty}e^{3v(t)}=0$ and thus $\lim_{t\to\infty} V_3(e^t)=0$. Hence we have again $\lim_{t\to\infty}C_{3,4}(e^t)=0=\lim_{t\to\infty}v'(t)$.\\

Similarly, if $t\to -\infty$, then $V_4(e^t)$ and $V_3(e^t)$ both approach zero, which can be seen as follows. Since $V_4(r)$ is a finite integral (i.e. $e^{4w(r)}$ is in $L^1$ when integrating over the origin), $V_4(r)$ must tend to zero if $r\to 0$, or equivalently  $V_4(e^t)$ tends to zero as $t\to-\infty$. Then, from formula \eqref{eq.Vs3}, we obtain that $e^{v(t)}$ converges to $0$ as $t\to -\infty$ (as otherwise $V_4(e^t)$ would be infinite). Thus, again by \eqref{eq.Vs3}, we obtain also $V_3(e^t)\to 0$ as $t\to -\infty$. The result then follows also in this case from L'H\^{o}pital's rule as above.\\

We have thus proved Lemma \ref{lemma.limits} for $C_{3,4}(e^t)$ and thus, using \eqref{eq.Cs2} and \eqref{eq.Candv}, for all the remaining cases.
\end{proof}


\section{Generalised Normal Metrics}\label{sect.normal}

In this section, we will define generalised normal metrics on $\RR^4\setminus\{0\}$ as a generalisation of $2$-dimensional complete normal metrics given by Finn \cite{F65}. In \cite{N12} a correspondence between finite area singular points and complete ends was discovered. We use this duality to define a finite area singular point. 

\begin{defn}[Generalised normal metrics]\label{defn.normal}
Suppose that $e^{2w}\abs{dx}^2$ is a metric on $\RR^4 \setminus \{0\} $ with finite total Paneitz $Q$ curvature 
\begin{equation*}
\int_{\RR^4} \abs{Q_g} e^{4w} dx < \infty. 
\end{equation*}   
We call $g$ a \emph{generalised normal metric}, if $w$ has the expansion
\begin{equation}\label{eq.defnormal}
w(x)=\frac{1}{4\pi^2}\int_{\RR^4}\log\Big(\frac{\abs{y}}{\abs{x-y}}\Big)\, Q_g(y)\,e^{4w(y)}\,dy+\alpha\log\abs{x}+C
\end{equation}
for some constants $\alpha,C\in\RR$. For such a generalised normal metric, we then define the \emph{averaged metric} $\bar{g}=e^{2\bar{w}}\abs{dx}^2$ by
\begin{equation}
\bar{w}(r):=\dashint_{\dB_r(0)}w(x)\,d\sigma(x) = \frac{1}{\abs{\dB_r(0)}}\int_{\dB_r(0)}w(x)\,d\sigma(x).
\end{equation}
Clearly, $\bar{g}$ is a rotationally symmetric metric.
\end{defn}

\begin{remark} 
Note that if we have a finite area metric, that is if $\int_{B_R(0)} e^{4w(y)} dy<\infty$, this implies that $\alpha >-1$.  
\end{remark}

The main theorem of this section is the following.

\begin{thm}\label{thm.sec3}
Let $g$ be a generalised normal metric on $\RR^4\setminus\{0\}$ with averaged metric $\bar{g}$ and define the mixed volumes $V_k$ (with respect to $g$) and $\bar{V}_k$ (with respect to $\bar{g}$) as in Definition \ref{defn.ratios}. Then
\begin{align}
V_3(r)&=\bar{V}_3(r)(1+\eps(r)),\label{eq.thm3.V3}\\
\tfrac{d}{dr}V_4(r)&=\tfrac{d}{dr}\bar{V}_4(r)(1+\eps(r)),\label{eq.thm3.V4}
\end{align}
where $\eps(r)\to 0$ if either $r\to 0$ or $r\to\infty$. Moreover, if the two limits
\begin{align*}
\lim_{r\to 0} \left (1 + r \frac{\partial \bar w}{\partial r }\right) \qquad\text{ and }\qquad \lim_{r\to \infty} \left (1 + r \frac{\partial \bar w}{\partial r }\right)
\end{align*}
both exist and are positive, then we have in addition that
\begin{align}
V_1(r)&=\bar{V}_1(r)(1+\eps(r)),\label{eq.thm3.V1}\\
V_2(r)&=\bar{V}_2(r)(1+\eps(r)),\label{eq.thm3.V2}
\end{align}
where again $\eps(r)\to 0$ if $r\to 0$ or $r\to\infty$.
\end{thm}

In order to prove this theorem, we start with the following technical lemma.

\begin{lemma}\label{lemma.1}
Suppose that the metric $e^{2w}\abs{dx}^2 $ on $\RR^4 \setminus\{0\}$ is a generalised normal metric. Then for any number $k>0$ we have that 
\begin{align}\label{eq.lemma1}
\dashint_{\dB_r(0)} e^{kw} d\sigma(x) = e^{k\bar w(r)} e^{o(1)}
\end{align}
where $o(1)\to 0$ as $r\to 0$ and as $r\to\infty$.
\end{lemma}

\begin{proof}
The proof for $r\to\infty$ was essentially covered in \cite[Lemma 3.2]{CQY1}. Note that there the formula \eqref{eq.lemma1} is proved for \emph{normal metrics} which differ from our definition of \emph{generalised} normal metrics by our additional term $\alpha\log\abs{x}$ in \eqref{eq.defnormal}. But this additional term, the fundamental solution of the bi-Laplacian, is rotationally symmetric and thus in equation \eqref{eq.lemma1}, $e^{\alpha \log\abs{x}}$ appears on both sides and hence cancels. For this reason, we only need to prove the lemma for $r\to 0$.\\

Suppose that $e^{2w}\abs{dx}^2$ is a generalised normal metric. As in the last section, we denote $F(y)=2Q_g(y)e^{4w(y)}$ which by assumption is in $L^1$. Then, splitting up $\RR^4$ into three regions, we have
\begin{align*}
w(x) &= \frac{1}{8\pi^2} \int_{B_{\abs{x}/2}(0)} \log \left( \frac{\abs{y}}{\abs{x-y}}\right)F(y) dy\\
&\quad + \frac{1}{8\pi^2} \int_{\RR^4 \setminus B_{3\abs{x}/2}(0)} \log \left( \frac{\abs{y}}{\abs{x-y}}\right)F(y) dy\\
&\quad +\frac{1}{8\pi^2} \int_{B_{3\abs{x}/2}(0)\setminus B_{\abs{x}/2}(0)} \log \left( \frac{\abs{y}}{\abs{x-y}} \right)F(y) dy + \alpha \log\abs{x} + C \\
& = w_1(x) + w_2(x) + w_3(x) + \alpha \log\abs{x} + C.
\end{align*}

As above, $\alpha\log\abs{x}+C$ is rotationally symmetric and thus in equation \eqref{eq.lemma1}, $e^{\alpha \log\abs{x}+C}$ appears on both sides and hence cancels. Therefore, we need only concentrate on $w_1(x)$, $w_2(x)$ and $w_3(x)$. We first consider $w_1(x)$, which we can rewrite as follows
\begin{align*}
w_1(x)&=\frac{1}{8\pi^2} \int_{\abs{y}\leq\abs{x}/2} \log \left( \frac{\abs{y}}{\abs{x}}\right)F(y) dy + \frac{1}{8\pi^2} \int_{\abs{y}\leq\abs{x}/2} \log \left( \frac{\abs{x}}{\abs{x-y}}\right)F(y) dy\\
&=f(\abs{x}) + w_1^0(x).
\end{align*}
As $f(\abs{x})$ is another rotationally symmetric term, we need to study only $w_1^0(x)$. In order to do this, let $\eta<\frac{1}{2}$, and estimate 
\begin{equation*}
\abs{w_1^0(x)} \leq C \left( \int_{\abs{y}\leq \eta\abs{x}} \left| \log \frac{\abs{x}}{\abs{x-y}}\right| \abs{F(y)} dy + \int_{\eta\abs{x}\leq\abs{y}\leq \frac{1}{2} \abs{x}} \left| \log \frac {\abs{x}}{\abs{x-y}}\right| \abs{F(y)} dy\right).
\end{equation*}

For the first integral, note that $\abs{y}\leq\eta\abs{x}$ implies
\begin{equation*}
(1-\eta)\abs{x}\leq \abs{x}-\abs{y}\leq\abs{x-y}\leq\abs{x}+\abs{y}\leq (1+\eta)\abs{x},
\end{equation*}
and hence
\begin{equation*}
\frac{1}{1+\eta}\leq\frac{\abs{x}}{\abs{x-y}}\leq \frac{1}{1-\eta}. 
\end{equation*}
This yields 
\begin{equation*}
\left| \log\frac{\abs{x}}{\abs{x-y}}\right| \leq \max \left\{ \left| \log \frac{1}{1+\eta} \right|, \left| \log \frac{1}{1-\eta}\right| \right\} = \log \frac{1}{1-\eta}.
\end{equation*}
In order to estimate the second integral, we use the bound $\abs{y}\leq\frac{1}{2}\abs{x}$, which by an analogous argument as above yields
\begin{equation*}
\left| \log\frac{\abs{x}}{\abs{x-y}}\right| \leq \log \frac{1}{1-\frac{1}{2}} = \log 2.
\end{equation*}
Combining these estimates and using $\int_{\RR^4}\abs{F(y)}dy < \infty$, we obtain
\begin{align*}
\abs{w_1^0(x)} \leq C \log \frac{1}{1-\eta} + \log 2 \int_{\eta \abs{x} \leq \abs{y} \leq \frac{1}{2}\abs{x}} \abs{F(y)}dy.
\end{align*}
For $\abs{x}\to 0$ and $\eta\to 0$, both terms above tend to zero, using again $\int_{\RR^4}\abs{F(y)}dy < \infty$. This proves that
\begin{equation}\label{eq.proofstep1}
w_1^0(x) = o(1), \qquad \text{as $\abs{x}\to 0$}.
\end{equation} 

As a second step, we estimate $w_2(x)$ by an argument which is dual to what we have just done. For $\eta>\frac{3}{2}$, we write 
\begin{equation*}
\abs{w_2(x)} \leq C \left( \int_{\abs{y}\geq \eta\abs{x}} \left| \log \frac{\abs{y}}{\abs{x-y}}\right| \abs{F(y)} dy + \int_{\eta\abs{x}\geq\abs{y}\geq \frac{3}{2} \abs{x}} \left| \log \frac {\abs{y}}{\abs{x-y}}\right| \abs{F(y)} dy\right).
\end{equation*}

To bound the first term, we notice that the inequality $\abs{y}\geq \eta\abs{x}$ implies
\begin{equation*}
(1-\tfrac{1}{\eta})\abs{y}\leq \abs{y}-\abs{x}\leq\abs{x-y}\leq\abs{y}+\abs{x}\leq (1+\tfrac{1}{\eta})\abs{y},
\end{equation*}
so that
\begin{equation*}
\frac{1}{1+\frac{1}{\eta}}\leq\frac{\abs{y}}{\abs{x-y}}\leq \frac{1}{1-\frac{1}{\eta}}. 
\end{equation*}
This gives 
\begin{equation*}
\left| \log\frac{\abs{y}}{\abs{x-y}}\right| \leq \max \left\{ \left| \log \frac{\eta}{\eta-1} \right|, \left| \log \frac{\eta}{\eta+1}\right| \right\} = \log \frac{\eta}{\eta-1}.
\end{equation*}
Similarly, we estimate the second term, with $\eta$ now replaced by $\frac{3}{2}$, which gives
\begin{equation*}
\left| \log\frac{\abs{y}}{\abs{x-y}}\right| \leq \log \frac{\frac32}{\frac32-1} = \log 3.
\end{equation*}
Combining these estimates and using $\int_{\RR^4}\abs{F(y)}dy < \infty$, we obtain
\begin{align*}
\abs{w_2(x)} \leq C \log \frac{\eta}{\eta-1} + \log 3 \int_{\eta \abs{x} \geq \abs{y} \geq \frac{3}{2}\abs{x}} \abs{F(y)}dy.
\end{align*}
For $\abs{x}\to 0$, we then send $\eta\to\infty$ slow enough such that $\eta\abs{x}\to 0$, in which case both terms above tend to zero, using again $\int_{\RR^4}\abs{F(y)}dy < \infty$. This proves that
\begin{equation}\label{eq.proofstep2}
w_2(x) = o(1), \qquad \text{as $\abs{x}\to 0$}.
\end{equation} 
As a third step, we consider the term 
\begin{align*}
\bar{w}_3(r)&=\dashint_{\dB_r(0)} w_3(x) d\sigma(x)\\
&= \frac{1}{8\pi^2} \int_{\frac{r}{2}\leq \abs{y} \leq \frac{3r}{2}} \left( \dashint_{\dB_r(0)} \log \frac{\abs{y}}{\abs{x -y}} d\sigma(x) \right)F(y) dy\\
&=\frac{1}{8\pi^2} \int_{\frac{r}{2}\leq  \abs{y} \leq \frac{3r}{2}} I(y) F(y) dy
\end{align*}
We claim that the inner integral
\begin{align*}
\abs{I(y)} &=\left| \dashint_{\dB_r(0)} \log \frac{\abs{y}}{\abs{x -y}} d\sigma(x) \right| \\
&\leq \frac{1}{\abs{\dB_r(0)}} \int_{\dB_r(0)\setminus \{x:\abs{x-y} < \frac{1}{3} \abs{y}\}} \left| \log \frac{\abs{y}}{\abs{x-y}} \right| d\sigma(x)\\
&\quad +\frac{1}{\abs{\dB_r(0)}} \int_{\dB_r(0)\cap \{x:\abs{x-y} < \frac{1}{3} \abs{y}\}} \left| \log \frac{\abs{y}}{\abs{x-y}} \right| d\sigma(x)\\
& = I_1(y) + I_2(y)
\end{align*}
is uniformly bounded on the annular region $\frac{r}{2}\leq \abs{y}\leq \frac{3r}{2}$. To estimate $I_1(y)$, note that on the region over which we integrate we have
\begin{equation*}
\frac{1}{3} \abs{y}\leq\abs{x-y}\leq\abs{x}+\abs{y}\leq 3\abs{y},
\end{equation*}
giving $\frac{1}{3}\leq \frac{\abs{y}}{\abs{x-y}}\leq 3$ and hence $I_1(y) \leq \log 3$. For the second integral, we have
\begin{align*}
I_2(y) \leq \frac{1}{\big| \dB_{\frac{r}{\abs{y}}}(0)\big|} \int_{\dB_{\frac{r}{\abs{y}}}(0)\cap \left\{x:\left |x -\frac{y}{\abs{y}} \right| < \frac{1}{3}\right\}} \left| \log \frac{1}{\abs{x-\frac{y}{\abs{y}}}} \right| d \sigma(x).
\end{align*} 
Now, as $\frac{r}{2} \leq \abs{y} \leq \frac{3}{2} r$, we have $\frac 23 \leq \frac {r}{\abs{y}} \leq 2$. Thus, as $ \log \frac{1}{\left| x-\frac{y}{\abs{y}} \right|}$ is integrable, we see that $I_2(y)$ is uniformly bounded as well, and hence $\abs{I(y)}$ is uniformly bounded for $\frac{r}{2}\leq \abs{y}\leq \frac{3r}{2}$ as claimed. The assumption of finite total $Q$ curvature then implies
\begin{equation}\label{eq.proofstep3a}
\dashint _{ \dB_r(0)} w_3 (x) d \sigma(x) = o(1), \quad \text{ as $ r\to 0$}. 
\end{equation}
Using \eqref{eq.proofstep1}, \eqref{eq.proofstep2} and \eqref{eq.proofstep3a} as well as Jensen's inequality, we obtain for $r\to 0$
\begin{equation}\label{eq.proofstep3b}
\begin{aligned}
k\, \dashint _{\dB_r (0)} w(x)d \sigma(x) &= \dashint _{\dB_r (0)} k \big(w(x) - w_3(x) \big) d \sigma(x)+ o(1)\\
&= \log \left( \dashint _{\dB_r(0)}e^{k(w(x) - w_3(x))} d\sigma(x)\right) + o(1).
\end{aligned}
\end{equation}  
Finally as a fourth and last step, we estimate the term 
\begin{align*}
\dashint_{\dB_r(0)} (e^{w_3(x)} -1) d\sigma(x)  = \dashint_{\dB_1(0)} (e^{w_3(rx)} - 1)d \sigma(x).
\end{align*}
Following \cite{F65}, this can be done by estimating $E_M = \{ \sigma \in \Sph ^ 3 : \abs{w_3(r \sigma)} > M \} $. Similar to the above, we have 
\begin{align*}
M \cdot \abs{E_M} &\leq \int_{E_M} \abs{w_3} d\sigma \leq \frac{1}{8\pi^2} \int_{B_{3r/2} (0)\setminus B_{r/2} (0)}
\left( \int _{ E_M} \left| \log \frac {\abs{y}} {\abs{r \sigma - y}} \right| d\sigma \right) \abs{F(y)}d y  \\
 &= \frac{1}{8\pi^2 } \int_{\frac {r}{2} \leq \abs{y} \leq \frac {3r}{2}} J(y) \abs{F(y)} dy.
\end{align*} 
As above, we have to estimate the inner integral
\begin{align*}
J(y) &= \int_{E_M \setminus \left \{\sigma : \abs{r \sigma - y} \leq \frac{\abs{y}}{3} \right\}} \left| \log \frac {\abs{y}}{\abs{r \sigma - y}}  \right| d\sigma + \int _{ E_M \cap\left \{\sigma : \abs{r \sigma - y} \leq \frac{\abs{y}}{3} \right \}}  \left| \log \frac {\abs{y}}{\abs{r \sigma - y}}  \right| d\sigma\\
  & =J_1(y) + J_2(y). 
\end{align*}
Clearly we have the estimate $J_1(y) \leq \log 3 \cdot \abs{E_M}$. We estimate the term $J_2(y)$ as follows. Observe that if we have 
$\abs{r \sigma - y} \leq \frac {\abs{y}}{ 3}$ then 
\begin{align*}
\log \bigg| \frac {\abs{y}}{\abs{r \sigma -y}} \bigg| \leq  \bigg| \log \frac{\abs{y}}{r} \bigg| + \bigg| \log \Big| \sigma - \frac{y}{r} \Big| \bigg| \leq  \log \frac 3 2 + \bigg| \log \Big| \sigma - \frac{y}{r} \Big| \bigg|.
\end{align*} 
We can thus bound $J_2(y)$ by the situation where $E_M$ is a $3$-dimensional disc centred at $\frac{y}{r}$ orthogonal to $y$, in which case we obtain
\begin{align*}
J_2(y) \leq C \abs{E_M} + C \abs{E_M} \log \frac {1}{\abs{E_M}} \leq C \left ( 1 + \log \frac {1}{\abs{E_M}} \right)\abs{E_M}.
\end{align*}
Combining these estimates, we get
\begin{align*}
 M \leq o(1) \left ( 1 + \log \frac{1}{\abs{E_M}} \right),
\end{align*}
where $ o(1) \to 0$ as $ r \to 0$. This implies
\begin{equation*}
\abs{E_M} \leq C e^ {-M / o(1)},
\end{equation*}
and thus
\begin{equation}\label{eq.proofstep4}
\left| \dashint_{\dB_r(0)} (e^{kw_3(x)} -1) d\sigma(x) \right| = \frac {k}{\abs{\dB_ 1(0)}} \int_{-\infty}^{+\infty} (e^{kM} -1) \abs{E_M} dM = o(1),
\end{equation}
as $r \to 0$. Combining the estimates \eqref{eq.proofstep3b} and \eqref{eq.proofstep4}, we have 
\begin{align*}
k\bar{w}(r)= k\; \dashint_{\dB_r(0)} w(x) d \sigma(x) = \log \left(\dashint_{\dB_ r(0)} e^{kw(x)} d \sigma(x)\right) + o(1),
\end{align*}
which is equivalent to \eqref{eq.lemma1}, thus finishing the proof of Lemma \ref{lemma.1}
\end{proof}

We continue with a second technical lemma.

\begin{lemma}\label{lemma.2}
Suppose that $e^{2w}\abs{dx}^2$ is a generalised normal metric on $\RR^4 \setminus \{0\}$. Then, we have that
\begin{align}\label{eq.lemma2a}
\dashint_{ \dB_r (0)} \left ( \frac {\partial w } { \partial r }\right) ^k d \sigma (x) = O \left ( \frac { 1 }{ r ^ k } \right ) , \quad \text{ for $ k = 1,2,3$,}  
\end{align}
and 
\begin{align}\label{eq.lemma2b}
\dashint _{ \dB_r ( 0)} \left( \frac { \partial w }{ \partial r } \right ) ^ 2 d \sigma(x) = \left ( \frac {\partial \bar w} {\partial r } \right)^2\!\!(r) + o \left( \frac {1}{r^2} \right),
\end{align}
for $r \to 0$ as well as for $r\to\infty$.
\end{lemma}

\begin{proof}
The case $r\to\infty$ was essentially covered by Chang-Qing-Yang \cite[Lemma 3.4]{CQY1}, where the result is proved for normal metrics whose conformal factor $w(x)$ differs from ours -- as above -- by our additional term $\alpha\log\abs{x}$ in \eqref{eq.defnormal}. However, it is easy to extend their result to our situation, noting that $\alpha\log\abs{x}$ is rotationally symmetric with derivative $\partial_r(\alpha\log\abs{x})=\frac{\alpha}{r}$. We thus have to prove the lemma only for $r\to 0$.\\

First, observe that we have 
\begin{align*}
\frac{\partial \bar w}{\partial r}(r) = \dashint_{\dB_r(0)} \frac{\partial w}{\partial r} d\sigma (x).
\end{align*}
Using our definition of generalised normal metrics, we see that
\begin{align*}
\frac{\partial w}{\partial r}(x) = \int_{\RR^4} -K(x,y) F(y) dy + \frac{\alpha}{r}
\end{align*}
where $K(x,y) =\frac{1}{8\pi^2} \partial_r \log \abs{x - y}$ and $F(y)= 2Q_g(y)e^{4w(y)}$ as before. Recall that $F(y)$ is integrable over $\RR^4$ by our assumption on finite total $Q$ curvature. Computing $K(x,y)$, we get that 
\begin{equation*}
\partial_r \frac{1}{8\pi^2} \log \abs{x-y} = \frac{1}{8\pi^2}\left \langle \frac{\nabla_x \abs{x-y}}{\abs{x -y}}, \frac{x}{\abs{x}} \right\rangle = \frac{1}{8\pi^2} \left \langle \frac{x - y}{\abs{x - y}^2} , \frac{x}{\abs{x}} \right \rangle  =\frac{1}{8\pi^2} \frac{\abs{x}^2 - x \cdot y}{\abs{x}\abs{x - y}^2}.
\end{equation*} 
Then, we have that 
\begin{align*}
\bigg| \, \dashint_{ \dB_r(0)} &\left( \frac{\partial w}{\partial r} \right)^k d \sigma(x)\,\bigg| \\
&\leq C \Bigg( \int_{\RR^4} \bigg(\; \dashint_{\dB_r(0)}\abs{K(x,y)}^k d \sigma(x) \bigg) \abs{F(y)} dy  \Bigg) \Bigg( \int_{\RR^4} \abs{F(y)} dy \Bigg)^{k -1} \!\!+ C \frac {\abs{\alpha}^k}{r^k}.
\end{align*}
Hence, it suffices to prove that for $k=1,2,3$,
\begin{align*}
\dashint_{\dB_r(0)} \abs{K(x,y)}^k d\sigma(x)= O \left (\frac{1}{r^k} \right), \qquad \forall y \in \RR ^ 4.
\end{align*}
In order to verify this estimate, we use the following 
\begin{equation}\label{eq.x2xy}
2(\abs{x}^2 - x \cdot y) = \abs{x-y}^2 + \abs{x}^2 - \abs{y}^2. 
\end{equation}
Then we need only verify that for all $ y \in \RR^4$ 
\begin{equation}\label{eq.unifC}
\dashint_{\dB_r(0)} \frac{\left| \abs{x}^2 -\abs{y}^2 \right|^3 }{\abs{x -y}^6} d\sigma(x) \leq C 
\end{equation}
for some constant independent of $y$. Using the homogeneity of the integrand, we only need to consider this integral for $r=1$. Hence, assuming $\abs{x}=1$, if we suppose for $\delta > 0$ that $\abs{y} \leq 1 - \delta$ or $\abs{y} \geq 1 - \delta$ then $\abs{\abs{x}-\abs{y}} \leq \abs{x-y}$ implies
\begin{align*}
\frac{1}{\abs{x - y}} \leq \frac{1}{\abs{1 - \abs{y}}} 
\end{align*} 
so that 
\begin{equation*}
\frac{\abs{\abs{x}^2 -\abs{y}^2}^3}{\abs{x - y}^6} \leq \frac{\abs{1 - \abs{y}^2}^3}{\abs{1 - \abs{y}}^ 6} \leq \frac{\abs{1 + \abs{y}}^3}{\abs{1 - \abs{y}}^3} \leq C(\delta).
\end{equation*}
This yields \eqref{eq.unifC} for these $y$. Moreover, for $\abs{y} \in (1-\delta,1) \cup (1 , 1+ \delta)$, we have the calculus inequality
\begin{align*}
\int_{\dB_1(0)} \frac{d\sigma(x)}{\abs{x -y}^6} \leq \frac{C}{\abs{1 -\abs{y}}^3},
\end{align*}
which yields \eqref{eq.unifC} also for the remaining cases and hence establishes \eqref{eq.lemma2a}.\\

To prove \eqref{eq.lemma2b}, we note that $ \frac{1}{\abs{x - y}^2}$ is the Green's function of the Laplacian on $\RR^4$, so
\begin{align*}
\Lap \frac{1}{\abs{x - y}^2}= C \delta_y(x).
\end{align*}
If $\abs{y}\leq r$, this implies 
\begin{align*}
\int_{\dB_r(0)} \partial_r \frac{1}{\abs{x -y}^2}d\sigma(x) = \int_{B_r(0)} \Lap \frac{1}{\abs{x -y}^2} dx = \int_{B_r(0)} C \delta_y(x)dx = C.  
\end{align*}
Thus, we see that 
\begin{equation*}
\partial_r \; \dashint_{\dB_r(0)} \frac{1}{\abs{x -y}^2} d\sigma(x) = \dashint_{\dB_r(0)} \partial_r \frac{1}{\abs{x -y}^2} d\sigma(x) = \frac{C}{\abs{\dB_r(0)}} = \frac{C}{2\pi^2 r^3}.
\end{equation*}
Therefore, we must have
\begin{align*}
\dashint_{\dB_r(0)} \frac{1}{\abs{x -y}^2} d\sigma(x)= -\frac{C}{4\pi^2 r^2} + D,
\end{align*}
for some constants $C$ and $D$. As we have $\abs{y}\leq r$, we may let $y = 0$, which implies that
\begin{align*}
\frac{1}{r^2} = -\frac{C}{4\pi^2 r^2} + D,
\end{align*} 
hence $D = 0$ and $-\frac{C}{4\pi^2}= 1$. Similarly, if $\abs{y}\geq r$ we have that
\begin{align*}
\partial_r \; \dashint_{\dB_r(0)} \frac{1}{\abs{x -y}^2}d\sigma(x)= 0
\end{align*}
and hence 
\begin{align*}
\dashint_{\dB_r(0)} \frac{1}{\abs{x -y}^2}d\sigma(x)= C.
\end{align*}
We may then let $\abs{x}\to 0$ to conclude $C = \frac{1}{\abs{y}^2}$. Altogether, we have proved
\begin{equation}\label{eqn_integral}
\dashint_{\dB_r(0)} \frac{1}{\abs{x -y}^2}d\sigma(x) =
\begin{cases}
\frac{1}{r^2}, & \text{ if $\abs{y}\leq r$},\\ 
\frac{1}{\abs{y}^2}, & \text{ if $\abs{y}> r $}.
\end{cases}
\end{equation}
It follows that we can write
\begin{align*}
\frac{\partial \bar w}{\partial r}(r) = \int_{\RR^4} - \bar K(x,y) F(y) dy 
\end{align*}
where \eqref{eq.x2xy} yields 
\begin{equation*}
\bar K(x,y) = \dashint_{\dB_r(0)} \frac{\abs{x}^2 - x \cdot y}{8\pi^2\abs{x}\abs{x -y}^2} d\sigma(x) = \frac{1}{16\pi^2 r} \; \dashint_{\dB_r(0)} 1+ \frac{r^2-\abs{y}^2}{\abs{x-y}^2} d\sigma(x).
\end{equation*}
For $\abs{y}\leq r$, we then obtain from \eqref{eqn_integral} 
\begin{equation*}
\bar K(x,y) = \frac{1}{16\pi^2 r} \Big(1 + \frac{r^2-\abs{y}^2}{r^2}\Big) = \frac{1}{16\pi^2 r}\Big(2-\frac{\abs{y}^2}{r^2}\Big).
\end{equation*}
Moreover, for $\abs{y}>r$, \eqref{eqn_integral} yields
\begin{align*}
\bar K(x,y) = \frac{1}{16\pi^2 r} \Big(1 + \frac{r^2-\abs{y}^2}{\abs{y}^2}\Big) = \frac{1}{16\pi^2 r}\Big(\frac{r^2}{\abs{y}^2}\Big).
\end{align*}
Together, we have
\begin{equation}
\bar K(x,y) = 
\begin{cases} 
\frac{1}{16\pi^2 r} \left(2 - \frac{\abs{y}^2}{r^2} \right), &\text{ if $\abs{y}\leq r$},\\
\frac{r}{16\pi^2 \abs{y}^2}, &\text{ if $\abs{y} > r$}.
\end{cases}
\end{equation}
We can then estimate
\begin{align*}
\dashint_{\dB_ r(0)} &\left| \left( \frac{\partial w}{\partial r} \right)^2 - \left( \frac{\partial \bar w}{\partial r} \right)^2 \right| d\sigma(x) \\ 
& \leq \int_{\RR^4} \abs{F(y)} dy \cdot \int_{\RR^4} \bigg( \; \dashint_{\dB_r} \abs{K(x,y) - \bar K(x,y)}^2 d\sigma(x) \bigg) \abs{F(y)} dy 
\end{align*}
using
\begin{align*}
\abs{K(x,y) - \bar K(x,y)} =  
\begin{cases} 
\frac{1}{16\pi^2\abs{x}} \big| \frac{(\abs{y}^2 - \abs{x}^2)(\abs{x -y}^2 - \abs{x}^2)}{\abs{x - y}^2\abs{x}^2}  \big|, & \text{ if $\abs{y} < \abs{x}$}, \\
\frac{1}{16\pi^2\abs{x}} \big| \frac{(\abs{y}^2 - \abs{x}^2)(\abs{x -y}^2 - \abs{y}^2)}{\abs{x - y}^2\abs{y}^2}  \big|, & \text{ if $\abs{y} > \abs{x}$}.
\end{cases}
\end{align*}
Note first that
\begin{align*}
\dashint_{\dB_r(0)}\abs{K(x,y) - \bar K(x,y)}^2 d\sigma(x) = O \left( \frac{1}{\abs{x}^2} \right) , \quad \text{ as $\abs{x} \to 0$ and $\abs{y} \leq \abs{x}^{1/2}$}.
\end{align*}
Moreover, if $\abs{y} \geq \abs{x}^{1/2}$, which implies that $\abs{y} > \abs{x}$ if $\abs{x} \to 0$, we get that 
\begin{align*}
\abs{K(x,y) - \bar K(x,y)} \leq C \abs{x} \left| \frac{1}{\abs{x -y}^2} - \frac{1}{\abs{x}^2}\right| & \leq \frac{1}{16\pi^2 \abs{x}}\abs{y}^2 \left| \frac{1}{\abs{y}^2} - \frac{1}{\abs{x - y}^2} \right|\\
&\leq \frac{1}{16 \pi^2 \abs{x}} \left| 1- \frac{1}{\big| \frac{x} {\abs{y}} - \frac{y}{\abs{y}} \big|^2} \right|\\
&= o \left(\frac{1}{\abs{x}} \right).
\end{align*}
The last step follows from the fact that $\abs{y} \geq \abs{x}^{1/2}$ implies $\abs{x}^{1/2} \geq \frac{\abs{x}}{\abs{y}}\to 0$ as $\abs{x} \to 0$.
Finally, combining the various estimates, we have that 
\begin{align*}
&\dashint_{\dB_r(0)} \left| \left( \frac{\partial w}{\partial r} \right)^2 - \left( \frac{\partial \bar w}{\partial r} \right)^2 \right| d\sigma(x)\\
& = \int_{\abs{y} \leq \abs{x}^{1/2}} O \left( \frac{1}{\abs{x}^2} \right) \abs{F(y)} dy + \int_{\abs{y} \geq \abs{x}^{1/2}} o \left( \frac{1}{\abs{x}^2} \right) \abs{F(y)} dy = o \left( \frac{1}{\abs{x}^2} \right).\qedhere
\end{align*}
\end{proof}

Using the two technical lemmas, we can now prove Theorem \ref{thm.sec3}.

\begin{proof}[Proof of Theorem \ref{thm.sec3}]
The formulas \eqref{eq.thm3.V3} and \eqref{eq.thm3.V4} follow immediately from Lemma \ref{lemma.1}. It thus remains to prove \eqref{eq.thm3.V1} and \eqref{eq.thm3.V2}. From \eqref{eq.Vs2}, we know that
\begin{equation*}
V_2(r) =\frac{1}{4}\int_{\dB_r(0)}\Big(\frac{1}{r}+\frac{\partial w}{\partial r}\Big)e^{2w}\, d\sigma(x),
\end{equation*}
hence
\begin{align*}
V_2(r) - \bar V_2(r) &= \frac{1}{4} \left( \frac{1}{r} + \frac{\partial \bar w}{\partial r} \right) \int_{\dB_r(0)} (e^{2w} - e^{2\bar w}) d\sigma(x)\\
&\quad + \frac{1}{4} \int_{\dB_r(0)} \left( \frac{\partial w}{\partial r} - \frac{\partial \bar w}{\partial r} \right)(e^{2w} - e^{2\bar w}) d\sigma(x).
\end{align*}
Applying the previous Lemmas \ref{lemma.1} and \ref{lemma.2} we see that
\begin{align*}
V_2 (r) - \bar V_2(r) & =\bar V_2(r) o(1) + \bigg( \int_{\dB_r(0)} \! \Big( \frac{\partial w}{\partial r} - \frac{\partial \bar w}{\partial r} \Big)^2 \! d\sigma(x) \bigg)^{\! 1/2} \! \bigg( \int_{\dB_r(0)} \! (e^{2w} - e^{2\bar w})^2 d\sigma(x) \bigg)^{\! 1/2}\\
& = \bar V_2(r) o(1) + \abs{\dB_r(0)} \cdot \frac{1}{r}\cdot e^{2\bar w} o(1),
\end{align*}
which implies \eqref{eq.thm3.V2} under the assumption that $ \lim_{r\to 0} (1 + r \frac{\partial \bar w}{\partial r}) > 0$ or $ \lim_{r\to \infty} (1 + r \frac{\partial \bar w}{\partial r}) > 0$, respectively. Similarly, using again \eqref{eq.Vs2}, we have that 
\begin{equation*}
V_1(r) =\frac{1}{4}\int_{\dB_r(0)}\Big(\frac{1}{r}+\frac{\partial w}{\partial r}\Big)^2e^{w}\, d\sigma(x),
\end{equation*}
and thus
\begin{align*}
V_1(r) - \bar V_1(r) &= \frac{1}{4} \left(\frac{1}{r} + \frac{\partial \bar w}{\partial r} \right)^2 \int_{\dB_r(0)}(e^w - e^{\bar{w}}) d\sigma(x)\\
&\quad + \frac{1}{2r} \int_{\dB_r(0)} \left(\frac{\partial w}{\partial r} - \frac{\partial \bar w}{\partial r} \right)(e^w - e^{\bar{w}})d\sigma(x)\\
&\quad + \frac{1}{4} \int_{\dB_r(0)} \left( \left( \frac{\partial w}{\partial r} \right)^2 - \left(\frac{\partial \bar w}{\partial r}\right)^2 \right) e^w d\sigma(x).
\end{align*}
Applying the H\"older inequality and the Lemmas \ref{lemma.1} and \ref{lemma.2}, we get that 
\begin{align*}
V_1(r) - \bar V_1(r) = \bar V_1(r)o(1) + \abs{\dB_r(0)} \frac {1}{r^2} e^{\bar{w}} o(1),
\end{align*}
which implies \eqref{eq.thm3.V1} under the assumption that $\lim_{ r\to 0} ( 1 + r \frac{\partial \bar w}{\partial r}) > 0$ or $\lim_{ r\to \infty} ( 1 + r \frac{\partial \bar w}{\partial r}) > 0$, respectively. Hence the Theorem is proved.
\end{proof}

Another direct consequence of Lemma \ref{lemma.1} (with $k = 4$) is the following.
\begin{cor}\label{cor.compfinite}
Suppose that $e^{2w} \abs{dx}^2$ is a complete generalised normal metric on $\RR^4 \setminus\{ 0 \}$ with finite area over the origin. Then its averaged metric $e^{2 \bar w(r)} \abs{dx}^2$ is also a complete metric with finite area over the origin.
\end{cor}

Before proving our main theorems for generalised normal metrics, we need one more estimate. We note that in Lemma \ref{lemma.c23}, the condition of non-negative scalar curvature could be replaced with $v''(t)=O(1)$ as $ t\rightarrow \pm \infty$. We prove the corresponding estimate for a generalised normal metric and its symmetrisation.

\begin{lemma}\label{lemma.lap}
Let $e^{2w} \abs{dx}^2$ be a generalised normal metric on $\RR^4 \setminus\{0\}$. Then we have 
\begin{align*}
\Lap \bar w(r) \leq \frac{C}{r^2}
\end{align*}
for some constant $C\in\RR$.
\end{lemma}
\begin{proof}
We note that
\begin{align*}
\Lap \bar w(r) = \Lap \dashint_{\dB_r(0)} w(x) d\sigma(x) = \dashint_{\dB_r(0)} \Lap w(x) d \sigma(x).
\end{align*}
Thus as $g$ is a generalised normal metric, we obtain 
\begin{align*}
\Lap \bar w(r) = \frac{1}{4\pi^2} \int_{\RR^4} \left( 2\; \dashint_{\dB_r(0)} \frac{1}{\abs{x - y}^2} d\sigma(x) \right) Q_g(y) e^{4w(y)} dy + \frac{2\alpha}{r^2}.
\end{align*}
By equation \eqref{eqn_integral}, we see that 
\begin{align*}
\Lap \bar w(r) \leq \frac{1}{2\pi^2 r^2} \int_{\RR^4}\abs{Q_g(y)}e^{4w(y)} dy + \frac{2\alpha}{r^2} \leq \frac{C}{r^2},
\end{align*}
using the assumption of finite total $Q$ curvature in the last step. This establishes the lemma.
\end{proof}

Now, we write $\bar w(r) = \bar w(e^t)$ and $v(t) = \bar w(e^t) + t$ as in the previous section. Remember that we verified there that 
\begin{align*}
v''''(t)- 4 v''(t) = 2\; \dashint_{\dB_r(0)} Q_g e^{4v} d\sigma = F(t), \quad -\infty < t < \infty, 
\end{align*}
with 
\begin{align*}
\int_{- \infty}^{\infty} F(t) dt = \frac{2}{\abs{\Sph^3}} \int_{\RR^4} Q_g e^{4w} dx = \frac{1}{\pi^2} \int_{\RR^4} Q_g e^{4w} dx < \infty
\end{align*}
and 
\begin{align*}
\int_{- \infty}^{\infty}\abs{F(t)} dt \leq \frac{1}{\pi^2} \int_{\RR^4} \abs{Q_g}e^{4w} dx < \infty .
\end{align*}

We then get a proof of our main theorems for generalised normal metrics.

\begin{proof}[Proof of Theorem \ref{thm.R4} and \ref{thm.ratios} for generalised normal metrics]
Let $g$ be a generalised normal metric and write its averaged metric $\bar g = e^{2 \bar w(r)} \abs{dx}^2$ in spherical co-ordinates. Then, as we have seen in the previous section,
\begin{equation}\label{eq.319}
\lim_{t\to\infty}v'(t)-\lim_{t\to-\infty}v'(t)=-\frac{1}{4\pi^2} \int_{\RR^4}Q_{\bar{g}} e^{4\bar{w}}dx.
\end{equation}
Note that $\bar g =e^{2 \bar w}\abs{dx}^2$ is conformally flat and $\bar{w}(r) = \dashint_{\dB_r(0)} w(x) d\sigma(x)$, so that 
\begin{equation*}
2 Q_{\bar{g}} e^{4 \bar w}(r) = \Lap^2 \bar w(r) = \dashint_{\dB_r(0)} \Lap^2 w(x) d\sigma(x) =\dashint_{\dB_r(0)} 2Q_g e^{4w} d\sigma(x).
\end{equation*}
This then implies
\begin{equation}\label{eq.320}
\int_{\RR^4} Q_{\bar{g}} e^{4 \bar w} dx = \int_{\RR^4} Q_g e^{4w} dx. 
\end{equation}
Applying the asymptotic estimates of $ V_3 $ and $ V_4 $, we get that 
\begin{equation}\label{eq.321}
\lim_{r \to 0} \frac{\left( \int_{\partial B_r(0)} e^{3w} d\sigma (x) \right)^{4/3}}{\int_{B_r(0)} e^{4w} dx}  = \lim _{r\to 0 } \frac{\abs{\Sph^3}^{4/3} e^{4 \bar w} r^4}{V_4(r)}
\end{equation}
as well as
\begin{equation}\label{eq.322}
\lim_{r \to \infty} \frac{\left( \int_{\partial B_r(0)} e^{3w} d\sigma (x) \right)^{4/3}}{\int_{B_r(0)} e^{4w} dx}  = \lim _{r\to \infty } \frac{\abs{\Sph^3}^{4/3} e^{4 \bar w} r^4}{V_4(r)}.
\end{equation}
Plugging \eqref{eq.320}--\eqref{eq.322} into \eqref{eq.319}, using also Lemma \ref{lemma.limits}, we obtain a proof of Theorem \ref{thm.R4} in the special case of generalised normal metrics. The same reasoning proves Theorem \ref{thm.ratios} for this type of metrics, using Theorem \ref{thm.sec3} with equations \eqref{eq.thm3.V1}, \eqref{eq.thm3.V2}.  
\end{proof}

\section{Singularity Removal Theorem}\label{sect.general}

We show that every metric $g=e^{2w}\abs{dx}^2$ on $\RR^4\setminus\{0\}$ which is complete at infinity and has finite area over the origin and satisfies \eqref{eq.finiteQ} and \eqref{eq.nonnegR} is a generalised normal metric as in Definition \ref{defn.normal}. Therefore, by the previous section, the generalised Chern-Gauss-Bonnet formula (Theorem \ref{thm.R4}) as well as Theorem \ref{thm.ratios} holds for any such metric. 

\begin{theorem}\label{thm.singremove}
Suppose that the metric $(\RR^4 \setminus \{0\}, e^{2w} \abs{dx}^2)$ is a complete finite area metric with finite total $Q$ curvature, $\int_{\RR^4}\abs{Q_g}e^{4w} dx < \infty$, and scalar curvature non-negative at infinity and at the origin. Then it is a generalised normal metric.
\end{theorem}
\begin{proof} 
First, let us denote
\begin{align*}
v(x) = \frac{1}{4\pi^2} \int_{\RR^4} \log\frac{\abs{y}}{\abs{x -y}} Q_g(y) e^{4w(y)} dy,
\end{align*}
and define $\psi(x) = w(x) - v(x)$. We then note that $\Lap^2 \psi =0 $ on $\RR^4  \setminus \{0\}$. In particular, we see that $\Lap\psi$ is harmonic on $\RR^4 \setminus \{0\}$. The transformation formula for the scalar curvature gives 
\begin{align*}
\Lap w + \abs{\nabla w}^2 = -\frac{\R_g}{6} e^{2w},
\end{align*}
where $\R_g$ is the scalar curvature for the metric $g=e^{2w}\abs{dx}^2$. As $\Lap\psi$ is harmonic on $\RR^4 \setminus \{0\}$, by the mean value equality
\begin{align*}
\Lap \psi(x_0) & = \dashint_{\dB_r(x_0)} \Lap \psi(x) d\sigma(x)\\
& = - \; \dashint_{\dB_r(x_0)} \left(\abs{\nabla w}^2 + \frac{\R_g}{6} \right) d\sigma(x) - \dashint_{\dB_r(x_0)} \Lap v(x) d\sigma(x),
\end{align*} 
if $\dB_r(x_0)\subset \RR^4 \setminus\{0\}$. Hence, for sufficiently large and sufficiently small $\abs{x_0}$, we see that $\R_g \geq 0$ and therefore the first term is non-positive. For the second term, by an argument similar to the one from Lemma \ref{lemma.lap}, we have that
\begin{align*}
\left| \int_{\dB_r(x_0)} \Lap v(x) d\sigma(x) \right| &= \frac{1}{2\pi^2} \int_{\RR^4}  \left( \dashint_{\dB_1(0)} \frac{1}{\abs{r \sigma + x_0 - y}^2} d\sigma \right) Q_g(y) e^{4w(y)} dy\\
& \leq \frac{1}{2\pi^2 r^2} \int_{\RR^4} \abs{Q_g(y)} e^{4w(y)} dy.  
\end{align*}
Hence taking $r = \frac{\abs{x_0}}{2}$, we see that 
\begin{align*}
\Lap \psi(x_0) \leq \frac{C}{\abs{x_0}^2} 
\end{align*}
for all $x_0 \in \RR^4 \setminus \{ 0 \} $ and hence 
\begin{align*}
\Lap \left(\psi(x_0) + \frac{C}{2} \log \abs{x_0} \right) \leq 0. 
\end{align*}
This means that $\left(\psi(x) + \frac{C}{2} \log \abs{x} \right)$ is a sub-harmonic and biharmonic function on $\RR^4 \setminus \{0\}$, i.e. $\Lap \left( \psi(x) + \frac{C}{2} \log \abs{x} \right)$ is harmonic and non-positive. Then by B\^ocher's Theorem, see \cite[Thm 3.9]{A01}, it follows that 
\begin{align*}
\Lap \left( \psi(x) + \frac{C}{2} \log \abs{x} \right) = \beta \frac{1}{\abs{x}^2} + b(x)
\end{align*}
where $\beta\leq 0$ and $b(x)$ is a harmonic function on $\RR^4$. Hence, we find that 
\begin{align*}
h(x) = \psi(x) + \frac{C+\beta}{2} \log \abs{x} = w - v - \alpha \log \abs{x}
\end{align*}
is a biharmonic function on $\RR^4$ and hence smooth. We now show that $h$ is in fact constant. As $\Lap h$ is harmonic, we get by the mean value theorem
\begin{equation}\label{eqn_meanvalue}
\begin{aligned}
\Lap h(x_0) &= \dashint_{\dB_r(x_0)} \Lap h(x) d\sigma(x)\\
& = - \;\dashint_{\dB_r(x_0)} \left( \abs{\nabla w}^2 + \frac{\R_g}{6} \right) d\sigma - \dashint_{\dB_r(x_0)} \Lap v(x) d\sigma(x) - \frac{2 \alpha}{r^2}.
\end{aligned}
\end{equation} 
By an argument as above, we obtain 
\begin{align*}
\Lap h(x_0) \leq 0,
\end{align*}
which shows that $\Lap h =C_0$ for some non-positive constant by Liouville's theorem for harmonic functions. Therefore any partial derivative of $h$ is harmonic, $\Lap h_{x_i} = 0$. Using once more the mean value equality, we get
\begin{align*}
\abs{h_{x_i}(x_0)}^2 = \left| \dashint_{\dB_r(x_0)} h_{x_i} d\sigma\right|^2 \leq \frac{1}{\abs{\dB_r(x_0)}^2}\int_{\dB_r(x_0)} \abs{\nabla h}^2 d\sigma.
\end{align*}
We also have the estimate
\begin{align*}
\abs{\nabla h}^2 \leq 4 \abs{\nabla w}^2 + 4 \abs{\nabla v}^2 + \frac{4\alpha^2}{\abs{x}^2} = - 4 C_0 - \frac{2\R_g}{3} e^{2w} + 4 \abs{\nabla v}^2 + \frac{C}{\abs{x}^2}
\end{align*}
and
\begin{align*}
\abs{\nabla v}^2 \leq C \left( \int_{\RR^4} \frac{1}{\abs{x -y}^2} \abs{Q_g} e^{4w} dy \right) \left( \int_{\RR^4} \abs{Q_g} e^{4w} dy \right) \leq C \int_{\RR^4} \frac{1}{\abs{x -y}^2} \abs{Q_g} e^{4w} dy.
\end{align*}
Hence, we can conclude that for each $x_0 \in \RR^4$ 
\begin{align*}
\abs{h_{x_i}}^2 \leq - 4 C_0
\end{align*}
and, again by Liouville's theorem, the derivatives of $h$ are constant which implies that $ \Lap h = C _0= 0$. Finally, this implies that the partial derivatives of $h$ vanish and $h$ is a constant. 
\end{proof}

\section{Localised versions and the manifold case}\label{sect.local}

The goal of this section is to prove our main result, Theorem \ref{thm.manifold}. This is achieved by localising the result from Theorem \ref{thm.R4}, see Theorem \ref{thm.localend} and Theorem \ref{thm.localsing} below.\\

Recall from Chang-Qing \cite{CQ} that on a smooth and compact manifold \emph{with boundary}, the $4$-dimensional Chern-Gauss-Bonnet formula \eqref{eq.CGB} has to be corrected with an additional boundary term. It reads
\begin{equation}\label{eq.CGBboundary}
\chi(M)=\frac{1}{4\pi^2}\int_M \big(\tfrac{1}{8}\abs{W_g}_g^2+Q_g\big) dV_g + \frac{1}{4\pi^2}\int_{\partial M} \big(L_g+T_g)d\sigma_g,
\end{equation}
where analogous to the Weyl term, $L_g \, d\sigma_g$ is a point-wise conformal invariant vanishing on Euclidean space and $T_g$ is the boundary curvature invariant given by
\begin{equation*}
T_g:=-\frac{1}{12}\partial_N \R_g +\frac{1}{6}\R_g H - R_{akbk}A_{ab} + \frac{1}{9}H^3 - \frac{1}{3}\tr A^3 +\frac{1}{3}\tilde{\Lap}H,
\end{equation*}
where $A$ and $H$ denote the second fundamental form and the mean curvature of the boundary, respectively, $\partial_N$ denotes the unit inward normal derivative and $\tilde{\Lap}$ the boundary Laplacian. In \cite{CQY2}, Chang-Qing-Yang proved the following Theorem.

\begin{thm}[Local Chern-Gauss-Bonnet formula for an end, Theorem 1 in \cite{CQY2}]\label{thm.localend}
Suppose $(E,g)=(\RR^4\setminus B,e^{2w}\abs{dx}^2)$ is a complete conformal metric with non-negative scalar curvature at infinity and finite total $Q$ curvature. Then
\begin{equation}\label{eq.localend}
\frac{1}{4\pi^2}\int_{\dB}T_g e^{3w}d\sigma(x) -\frac{1}{4\pi^2}\int_{\RR^4\setminus B}Q_ge^{4w}dx=\lim_{r\to\infty}\, \frac{\vol_g(\dB_r(0))^{4/3}}{4(2\pi^2)^{1/3}\vol_g(B_r(0)\setminus B)}.
\end{equation}
\end{thm}

Localising our Theorem \ref{thm.R4}, we prove the following dual result for finite area singular points.

\begin{thm}[Local Chern-Gauss-Bonnet formula for a singular region]\label{thm.localsing}
Suppose that $(S\setminus\{p\},g)=(B\setminus \{0\},e^{2w}\abs{dx}^2)$ has finite area, non-negative scalar curvature near the origin, and finite total $Q$ curvature. Then
\begin{equation}\label{eq.localsing}
\frac{1}{4\pi^2}\int_{\dB}T_g e^{3w}d\sigma(x) +\frac{1}{4\pi^2}\int_{B}Q_ge^{4w}dx=\lim_{r\to 0}\, \frac{\vol_g(\dB_r(0))^{4/3}}{4(2\pi^2)^{1/3}\vol_g(B_r(0))}.
\end{equation}
\end{thm}
\begin{rem}
Note that in the above two theorems we work with the identical expression for $T_g$, i.e.~not reversing the orientation of the boundary and the normal $N$. For a conformal metric $g=e^{2w}\abs{dx}^2$ on $\RR^4\setminus\{0\}$, we obtain \eqref{eq.main} by subtracting \eqref{eq.localsing} from \eqref{eq.localend}.
\end{rem}
To prove Theorem \ref{thm.localsing} we slightly modify the three steps of the proof of Theorem \ref{thm.R4}. First, assume that $w$ is a radial function on $\RR^4\setminus B$, and denote, as in Section \ref{sect.symm}, $\abs{x}=r=e^t$ and $v=w+t$. Then we have
\begin{equation}
v''''-4v''=2Q_ge^{4v}=: F, \qquad -\infty< t \leq 0.
\end{equation}
Following Section \ref{sect.symm} almost verbatim (but changing integration boundaries from $+\infty$ to $0$), we obtain an explicit solution $f(t)$ of this ODE, satisfying
\begin{equation*}
\lim_{t\to -\infty} f'(t)=\frac{1}{8}\int_{-\infty}^0 F(x)dx
\end{equation*}
and
\begin{equation*}
\lim_{t\to -\infty} f''(t)=\lim_{t\to -\infty} f'''(t)=0.
\end{equation*}
As in \eqref{eq.v}, $v(t)$ is of the form
\begin{equation*}
v(t)=c_0+c_1 t+c_2 e^{-2t}+c_3 e^{2t}+f(t)
\end{equation*}
and (as in Lemma \ref{lemma.c23}), under the condition of non-negative scalar curvature near the origin, we obtain $c_2=0$ and thus in particular
\begin{equation}\label{eq.thirddev}
\lim_{t\to -\infty} v'''(t)=0.
\end{equation}
In \cite[Remark 3.1]{CQ}, Chang-Qing proved that in this special rotationally symmetric situation the boundary curvature $T_g$ satisfies
\begin{equation*}
T_g e^{3v}=-\frac{1}{2}v'''+2v'.
\end{equation*}
Plugging this into the Chern-Gauss-Bonnet formula with boundary, we have
\begin{align*}
0&=\frac{1}{4\pi^2}\int_{s\leq t\leq 0} Q_g e^{4v}dx + \frac{1}{4\pi^2}\int_{t=0} T_g e^{3v}d\sigma - \frac{1}{4\pi^2}\int_{t=s} T_g e^{3v}d\sigma\\
&=\frac{1}{4\pi^2}\int_{s\leq t\leq 0} Q_g e^{4v}dx + \frac{1}{4\pi^2}\int_{t=0} T_g e^{3v}d\sigma - \frac{1}{4\pi^2}\abs{\Sph^3}\Big({-\frac{1}{2}}v'''(s)+2v'(s)\Big).
\end{align*}
Taking the limit as $s\to -\infty$ and using \eqref{eq.thirddev}, we obtain
\begin{equation}\label{eq.56}
\frac{1}{4\pi^2}\int_{\dB}T_g e^{3w}d\sigma(x) +\frac{1}{4\pi^2}\int_{B}Q_ge^{4w}dx=\lim_{t\to -\infty}v'(t).
\end{equation}
The fact that 
\begin{equation}\label{eq.57}
\lim_{t\to-\infty}v'(t)=\lim_{t\to-\infty}C_{3,4}(e^t)=\lim_{r\to 0}\, \frac{\vol_g(\dB_r(0))^{4/3}}{4(2\pi^2)^{1/3}\vol_g(B_r(0))}
\end{equation}
then follows again exactly as in Lemma \ref{lemma.limits}, which plugged into \eqref{eq.56} proves Theorem \ref{thm.localsing} for the case of rotationally symmetric metrics $g$.\\

Next, we localise Definition \ref{defn.normal} and introduce a notion of generalised normal metrics on the punctured ball.

\begin{defn}[Generalised normal metrics on the punctured ball]\label{defn.locnormal}
Let $g=e^{2w}\abs{dx}^2$ be a metric on $B \setminus \{0\}$, where $B$ is the unit ball in $\RR^4$. Assume that $g$ has finite total $Q$ curvature,
\begin{equation*}
\int_{B} \abs{Q_g} e^{4w} dx < \infty. 
\end{equation*}   
We say that $g$ is a \emph{generalised normal metric}, if $w$ has the expansion
\begin{equation}
w(x)=\frac{1}{4\pi^2}\int_{B}\log\Big(\frac{\abs{y}}{\abs{x-y}}\Big)\, Q_g(y)\,e^{4w(y)}\,dy+\alpha\log\abs{x}+h(x)
\end{equation}
for some constants $\alpha\in\RR$ and a biharmonic function $h$ on $B$. (We would like to emphasise here that $h$ is defined and biharmonic on all of $B$, in particular also over the origin!) For such a generalised normal metric, we then define the \emph{averaged metric} as before by $\bar{g}=e^{2\bar{w}}\abs{dx}^2$, where
\begin{equation*}
\bar{w}(r):=\dashint_{\dB_r(0)}w(x)\,d\sigma(x).
\end{equation*}
\end{defn}

\begin{lemma}
Suppose that the metric $e^{2w}\abs{dx}^2 $ on $B \setminus\{0\}$ is a generalised normal metric in the sense of Definition \ref{defn.locnormal}. Then for any number $k>0$ we have that 
\begin{align}
\dashint_{\dB_r(0)} e^{kw} d\sigma(x) = e^{k\bar w(r)} e^{o(1)}
\end{align}
where $o(1)\to 0$ as $r\to 0$.
\end{lemma}

\begin{proof}
This is analogous to Lemma \ref{lemma.1} and the proof of Lemma \ref{lemma.1} goes through almost verbatim. We only need to modify the domains of integration (essentially changing $\RR^4$ to $B$) and take care of the additional function $h(x)$. However, as $h(x)$ is biharmonic (and thus smooth) over the origin, we have
\begin{equation*}
\lim_{r\to 0}\dashint_{\dB_r(0)} e^{kh(x)} d\sigma(x) = e^{k\bar{h}(0)} =\lim_{r\to 0}e^{k\bar{h}(r)}.
\end{equation*}
We leave the remaining minor modifications to the reader.
\end{proof}

As a direct consequence, analogous to the first part of Theorem \ref{thm.sec3}, we then see that the mixed volumes taken with respect to the two metrics $g$ and $\bar{g}$ satisfy
\begin{align}
V_3(r)&=\bar{V}_3(r)(1+\eps(r)),\label{eq.510}\\
\tfrac{d}{dr}V_4(r)&=\tfrac{d}{dr}\bar{V}_4(r)(1+\eps(r)),\label{eq.511}
\end{align}
where $\eps(r)\to 0$ as $r\to 0$. Moreover, as in Corollary \ref{cor.compfinite}, if $g$ is a finite area metric on $B\setminus\{0\}$, then so is the averaged metric $\bar{g}$. Finally, $\bar{g}$ is clearly rotationally symmetric by definition. Thus, $v=\bar{w}+t$ satisfies \eqref{eq.56} and \eqref{eq.57}, that is
\begin{equation}\label{eq.5bars}
\frac{1}{4\pi^2}\int_{\dB}T_{\bar{g}}e^{3\bar{w}}d\sigma(x) +\frac{1}{4\pi^2}\int_{B}Q_{\bar{g}}e^{4\bar{w}}dx=\lim_{t\to -\infty}\bar{v}'(t) = \lim_{r\to 0}\frac{\bar{V}_3^{4/3}(r)}{(\pi^2/2)^{1/3}\,\bar{V}_4(r)}.
\end{equation}
As
\begin{equation*}
Q_{\bar{g}}e^{4\bar{w}}(r)=\dashint_{\dB_r(0)} \frac{1}{2}\Lap^2\bar{w}\, d\sigma(x) = \dashint_{\dB_r(0)} \frac{1}{2}\Lap^2 w\, d\sigma(x) = \dashint_{\dB_r(0)}Q_ge^{4w}d\sigma(x)
\end{equation*}
as well as
\begin{equation*}
T_{\bar{g}}e^{3\bar{w}}(r)=P_3\bar{w}=\overline{P_3w}=\dashint_{\dB_r(0)}T_g e^{3w}d\sigma(x),
\end{equation*}
where $P_3$ denotes the boundary operator associated to the Paneitz operator (see e.g. \cite{CQY2}), we can drop all the bars on the left hand side of \eqref{eq.5bars}. Due to \eqref{eq.510} and \eqref{eq.511}, we can also drop the bars on the right hand side of \eqref{eq.5bars}. Hence, we have proved Theorem \ref{thm.localsing} for the special case of generalised normal metrics on the punctured ball.\\

Finally, Theorem \ref{thm.localsing} follows from the above combined with the following lemma.
\begin{lemma}
Suppose that $(S\setminus\{p\},g)=(B\setminus \{0\},e^{2w}\abs{dx}^2)$ has finite area, non-negative scalar curvature near the origin, and finite total $Q$ curvature. Then $g$ is a generalised normal metric as in Definition \ref{defn.locnormal}.
\end{lemma}

\begin{proof} 
The proof of this lemma is very similar to the first part of the proof of Theorem \ref{thm.singremove}. We first denote
\begin{align*}
v(x) = \frac{1}{4\pi^2} \int_{B} \log\frac{\abs{y}}{\abs{x -y}} Q_g(y) e^{4w(y)} dy,
\end{align*}
and define $\psi(x) = w(x) - v(x)$. We then note that $\Lap^2 \psi =0 $ on $B  \setminus \{0\}$. In particular, we see that $\Lap\psi$ is harmonic on $B \setminus \{0\}$ and thus by the mean value equality
\begin{align*}
\Lap \psi(x_0) & = \dashint_{\dB_r(x_0)} \Lap \psi(x) d\sigma(x)\\
& = - \; \dashint_{\dB_r(x_0)} \left(\abs{\nabla w}^2 + \frac{\R_g}{6} \right) d\sigma(x) - \dashint_{\dB_r(x_0)} \Lap v(x) d\sigma(x),
\end{align*} 
if $\dB_r(x_0)\subset B \setminus\{0\}$. Here we used the transformation formula for the scalar curvature\begin{align*}
\Lap w + \abs{\nabla w}^2 = -\frac{\R_g}{6} e^{2w}.
\end{align*}
For sufficiently small $\abs{x_0}$, we know that $\R_g \geq 0$ by assumption. Moreover, by an argument similar to the one from Lemma \ref{lemma.lap}, we have that
\begin{align*}
\left| \int_{\dB_r(x_0)} \Lap v(x) d\sigma(x) \right| &= \frac{1}{2\pi^2} \int_{B}  \left( \dashint_{\dB_1(0)} \frac{1}{\abs{r \sigma + x_0 - y}^2} d\sigma \right) Q_g(y) e^{4w(y)} dy\\
& \leq \frac{1}{2\pi^2 r^2} \int_{B} \abs{Q_g(y)} e^{4w(y)} dy.  
\end{align*}
Hence, taking for example $r = \frac{\abs{x_0}}{2}$, we see that 
\begin{align*}
\Lap \psi(x_0) \leq \frac{C}{\abs{x_0}^2} 
\end{align*}
for some constant $C$ and all $x_0 \in B \setminus \{ 0 \} $ and hence 
\begin{align*}
\Lap \left(\psi(x_0) + \frac{C}{2} \log \abs{x_0} \right) \leq 0. 
\end{align*}
This means that $\left(\psi(x) + \frac{C}{2} \log \abs{x} \right)$ is a sub-harmonic and biharmonic function on $B \setminus \{0\}$, i.e. $\Lap \left( \psi(x) + \frac{C}{2} \log \abs{x} \right)$ is harmonic and non-positive and thus by B\^ocher's Theorem, see \cite[Thm 3.9]{A01}, we conclude 
\begin{align*}
\Lap \left( \psi(x) + \frac{C}{2} \log \abs{x} \right) = \beta \frac{1}{\abs{x}^2} + b(x)
\end{align*}
where $\beta\leq 0$ and $b(x)$ is a harmonic function on $B$ (including over the origin). Hence we find that 
\begin{align*}
h(x) = \psi(x) + \frac{C+\beta}{2} \log \abs{x} = w - v - \alpha \log \abs{x}
\end{align*}
is a biharmonic function on $B$. We have thus proved \eqref{eq.defnormal} for $w$, meaning that $w$ is a generalised normal metric in the sense of Definition \ref{defn.locnormal}.
\end{proof}

It is now easy to prove our main theorem.

\begin{proof}[Proof of Theorem \ref{thm.manifold}]
According to Definition \ref{def.manifold}, a manifold as in Theorem \ref{thm.manifold} can be split up into a compact manifold $N$ with boundary and a finite number of ends and singular regions, 
\begin{equation*}
M=N\cup\Big(\bigcup_{i=1}^k E_i\Big)\cup\Big(\bigcup_{j=1}^{\ell} S_j\Big).
\end{equation*}
Applying the Chern-Gauss-Bonnet formula for manifolds with boundary, given in \eqref{eq.CGBboundary} to $N$, as well as using Theorem \ref{thm.localend} for each of the ends and Theorem \ref{thm.localsing} for each of the singular regions, we immediately obtain the claimed Chern-Gauss-Bonnet type formula \eqref{eq.mainCGBformula}, noticing that all the boundary terms cancel each other out (as they obviously all appear twice with opposite signs). This finishes the proof of the theorem.
\end{proof}

\section*{Acknowledgements}
Parts of this work were carried out during two visits of HN at Queen Mary University of London. He would like to thank the university for its hospitality. These visits have been financially supported by RB's Research in Pairs Grant from the London Mathematical Society as well as HN's AK Head Travelling Scholarship from the Australian Academy of Science. RB would also like to thank the EPSRC for partially funding his research under grant number EP/M011224/1.

\makeatletter
\def\@listi{%
  \itemsep=0pt
  \parsep=1pt
  \topsep=1pt}
\makeatother
{\fontsize{10}{11}\selectfont

\vspace{10mm}

Reto Buzano (M\"{u}ller)\\
{\sc School of Mathematical Sciences, Queen Mary University of London, London E1 4NS, United Kingdom}\\

Huy The Nguyen\\
{\sc School of Mathematical Sciences, Queen Mary University of London, London E1 4NS, United Kingdom}\\
\end{document}